\documentclass[reqno]{amsart}
\usepackage{enumerate}
\usepackage{bbm}
\usepackage[margin=1.25in]{geometry}
\usepackage{amsmath,amsthm,amscd,amssymb}
\usepackage{latexsym}
\usepackage{color}
\usepackage{todonotes}
\usepackage{mathrsfs}
\usepackage{float}

\title[Stationarity in gene expression networks]{Stationarity and inference in multistate promoter models of stochastic gene expression via stick-breaking measures}

\author{William Lippitt}
\address{Biostatistics, University of Colorado Anschutz Medical Campus, Aurora, CO 80045}
\email{william.lippitt@cuanschutz.edu}

\author{Sunder Sethuraman}
\address{Mathematics, University of Arizona, Tucson, AZ 85721}
\email{sethuram@math.arizona.edu}

\author{Xueying Tang}
\address{Mathematics, University of Arizona, Tucson, AZ 85721}
\email{xytang@math.arizona.edu}

\numberwithin{equation}{section}

\newcommand{\N}{{\mathbb N}}
\newcommand{\X}{{\mathfrak X}}

\newcommand{\T}{{\mathbf T}}
\newcommand{\M}{{\mathbf M}}
\newcommand{\p}[1]{{\mathcal P\left( #1\right)}}

\newcommand{\E}[1]{{\mathbf E\left[ #1\right]}}
\renewcommand{\P}{{\mathbf P}}

\newtheorem{theorem}{Theorem}[section]

\newtheorem{cor}[theorem]{Corollary}

\newtheorem{prop}[theorem]{Proposition}
\newtheorem{defi}[theorem]{Definition}

\newtheorem{ex}[theorem]{Example}

\begin{document}

\begin{abstract}
In a general stochastic multistate promoter model of dynamic mRNA/protein interactions, we identify the stationary joint distribution of the promoter state, mRNA, and protein levels through an explicit `stick-breaking' construction of interest in itself.
 This derivation is a constructive advance over previous work where the stationary distribution is solved only in restricted cases.
Moreover, the stick-breaking construction allows to sample directly from the stationary distribution, permitting inference procedures and model selection. In this context, we discuss numerical Bayesian experiments to illustrate the results.

\end{abstract}



\maketitle

\section{Introduction}

Relatively recent models of mRNA creation and degradation in cells incorporate the notion of a stochastic `promoter' which influences birth rates and serves as a surrogate to the complex underlying structure of chemical reactions.  In such models, the stationary distribution of mRNA levels is of interest, given in particular that now readings from cells can be taken.

The multistate promoter process is a more involved model than the simple birth-death process with constant rates in which the evolution is somewhat regular and the stationary distribution is Poisson.  In particular, observations in types of cells indicate that the production of mRNA in the multistate process can be `bursty' and the levels of mRNA in stationarity can have heavy, non-Poissonian tails \cite{1Albayrak}, \cite{12Herbach et al} and references therein.  In this respect, the multistate mRNA process can reproduce such a phenomenon, and is now receiving much attention as a possible complex yet tractable model \cite{5Coulon}, \cite{7Dattani}, \cite{Herbach}, \cite{13Innocenti}, \cite{22Zhou}, and references therein.

The general multistate promoter process is a pair evolution $(E, M)$ where the state of the promoter $E\in \X$ belongs to discrete finite or countably infinite set $\X$ and the level of mRNA $M\in \{0,1,2,\ldots\}$ is a nonnegative integer.  The dynamics of the pair is that when $E=i$, the birth rate of $M$ to $M+1$ is $\beta_i\geq 0$ and the death rate of $M$ to $M-1$ is $\delta M$, proportional to $M$ with $\delta>0$, degradation being modeled independent of the state $E$.  On the other hand, for fixed $M$, the promoter $E=i$ switches to a different state $E=j$ with rate $G_{i,j}\geq 0$.  The parameters $\beta=\{\beta_i\}_{i\in \X}$ and `generator' $G = \{G_{i,j}\}_{i,j\in \X}$, where $G_{i,i} = -\sum_{j\neq i} G_{i,j}$, completely specify the process.

In \cite{Cao}, \cite{Ham}, \cite{14Kim}, \cite{Peccoud}, \cite{Ham}, the stationary distribution for mRNA levels $M$ is identified for multistate processes when $\X=\{1,2\}$ as a scaled Beta-Poisson mixture.  More generally, in \cite{Herbach}, when $\X=\{1,2,\ldots, n\}$ and the generator $G$ is such that $G_{i,j} = \alpha_j$ is independent of $i\neq j$, it is shown that the stationary distribution is a scaled Dirichlet-Poisson mixture.  In the `refractory' case when $\beta_i>0$ for exactly one $i\in \X=\{1,2,\ldots, n\}$ and $G$ is now allowed to be general where $G_{i,j}$ may depend on $i$, \cite{Herbach} derives a scaled Beta product-Poisson mixture.  In \cite{Zhou-Liu}, a general hypergeometric formula is given for general $G$ generators.   Although a certain generating function of the stationary distribution in the general model is known to satisfy a PDE in terms of parameters \cite{Herbach}, a constructive solution for the stationary distribution is not known for the general multistate rates promoter process.

In this context, the first aim of this note is to consider a multistate promoter process with general class creation rates $\beta$ and promoter switching rates $G$ in the context of `stick-breaking'. 
We identify an explicit form of the stationary distribution of $(E,M)$ in terms of a `Markovian stick-breaking' mixture distribution, reminiscent of the stick-breaking form of the Dirichlet process used much in nonparametric Bayesian statistics. That `stick-breaking' would be involved in such a characterization in this mRNA dynamics context was unexpected.  We also state formulas for certain moments aided by the `stick-breaking' formulation.

A second aim is to conduct statistical inference based on synthetic stationarily distributed data, with a future view toward inference with respect to laboratory biological mRNA data.  This has been considered in the literature for certain multistate models \cite{12Herbach et al} and references therein; see also \cite{Lin-Buchler}.  In this respect, we exploit the stick-breaking form of the stationary distribution to perform the inference which seems to allow for good computation and error bounds.

Still a third aim of this work is to extend results from the multistate mRNA model to a general multistate mRNA model with {\it protein} interactions, namely a process $(E, M, P)\in \X\times \{0,1,2,\ldots\}^2$ where $(E,M)$ rates are the same and the rates of $P$ to $P+1$ is $\alpha M$ and $P$ to $P-1$ is $\gamma P$ where $\alpha,\gamma>0$.  We mention, the identification of the stationary distribution in such a model, even in this `non-feedback' network, was posed as an open problem in \cite{Herbach}; we leave to future work to consider ramifications in models with `feedback'.  Recently, in \cite{Choudhary}, protein interactions have been considered in the `refractory' case where $\beta_i>0$ for exactly one state $E=i$ in terms of P\'olya urn models.  We mention also previous work on on-off promoter models \cite{Shah}.  We identify here in the general setting the stationary distribution of $(E,M,P)$ in terms of the `stick-breaking' apparatus, in particular interestingly `clumped' versions, and discuss computation of moments.

\medskip

{\it Aim 1.}  The identification of the stationary distribution of $(E,M)$ connects interestingly to disparate models of inhomogeneous Markov chains of their own interest.  First, by a Poisson representation introduced in \cite{Gardiner} for chemical reaction models, one can associate a piecewise deterministic Markov process $X$ where  $X$ satisfies an ODE depending on the process $E$ and represents a mass action kinetics process with respect to the levels of the promoter `chemicals'.  Then, it turns out at stationarity $M|(E,X) \sim {\rm Poisson}(\beta \cdot X)$.  The question now is what is the distribution of $(E,X)$.

Next, consider a discrete time inhomogeneous Markov chain $\{M_n\}_{n\geq 0}$ on $\X$ where the transition kernel from time $n$ to $n+1$ is $(I+ G/n)$.  In a sense, $\{M_n\}_{n\geq 0}$ is a discrete time version of the process $E$.  In \cite{DLS} (see also \cite{DS}), the limit of the empirical distribution of this chain is identified as a Markov stick-breaking measure $\nu_G$.  

Finally, \cite{BC}, in the study of `freezing MC's', generalizing those in \cite{DS}, considered the piecewise deterministic Markov process $(E,X)$ (although its connection with mRNA dynamics was perhaps not known).  They showed at stationarity, $X$ has the law of the empirical distribution limit $\nu_G$, among other results.
In combination with the stick-breaking characterization in \cite{DLS}, this  leads to the stick-breaking Poisson mixture identification of $(E,M)$ (Theorem \ref{cormultistatestationaryconstruction}).  Related moment computations are given in Section \ref{moments_subsect}.  Also, the notion of the `identifiability' with respect to the mRNA level $M$ is discussed in Section \ref{identifiability_sect}.

\medskip

{\it Aim 2.} 
In Section \ref{sec:estimation}, we discuss parameter estimation and model selection under a Bayesian framework. We assume that observed data come from the stationary distribution of the associated mRNA model with unknown parameters $(\beta, G)$.
In Section \ref{subsec:estimation}, with priors placed on the parameters, we describe a Gibbs sampler to draw samples from the posterior distribution. The empirical posterior means are used as estimators of $(\beta, G)$. Given estimated parameters from several candidate models with different number of states or different sparsity structures in $(\beta, G)$, we discuss in Section \ref{sec:selection} how to use Bayesian Information Criterion (BIC) to select the model underlying the observed data.

The key step in both inference tasks is to evaluate the likelihood function (the probability of observing the data) which are not in a closed-form. We utilize truncations of the stick-breaking form of the stationary distribution to approximate the likelihood with Monte Carlo simulations. 
The discussed procedures in Sections \ref{subsec:estimation} and \ref{sec:selection} are applied to synthetic datasets with various choices of $(\beta, G)$ for $|\X| =$ 2 or 3. In our experiments, when the sample size is large, the model parameters can be estimated accurately and the underlying models can be selected correctly with high probability. 

\medskip

{\it Aim 3.}  The protein model $(E, M, P)$ mentioned earlier can be analyzed by writing the interactions in terms of $\tilde E = (E,M)$ and $\tilde M = P$, where now $\tilde E$ is in the role of being a `promoter' with respect to protein levels $\tilde M$.  Since the promoter state space $\X\times \{0,1,2,\ldots\}$ is not finite, direct application of results in \cite{DLS}, \cite{BC} may not be possible as the transition operator $(I-\tilde G/n)$ will not be stochastic, that is $\tilde G$ with respect to $\tilde E$ transitions will not be bounded.  The idea however in Section \ref{protein_sect} is to represent the stationary measure via `clumped' versions of associated stick-breaking measures introduced in \cite{DLS}, perhaps of interest in itself (Theorem \ref{protein_thm}).

\medskip

The plan of the paper is to introduce notation and definitions of stick-breaking measures and their `clumped' forms in Section \ref{def_sect}.  In Section \ref{main_sect}, we discuss the relationship between certain time-inhomogeneous Markov chains, stick-breaking measures, piecewise deterministic Markov processes and multistate mRNA promoter models and formulate Theorem \ref{cormultistatestationaryconstruction}; in Section \ref{moments_subsect}, some moments are computed, and in Section \ref{identifiability_sect}, identifiability of parameters is discussed.  In Section \ref{protein_sect}, we discuss models which incorporate protein interactions and state Theorem \ref{protein_thm} which is then shown in Section \ref{clumped_sect}.  In Section \ref{sec:estimation}, we discuss how to utilize the stick-breaking constructions to estimate model parameters based on data from the stationary distribution (Section 5.1) and how to perform model selection (Section 5.2). Then, in Section \ref{conclusion}, we conclude.

\section{Stick-breaking measure representations and other definitions}
\label{def_sect}

We first introduce notation on spaces and matrices used throughout the article in Section \ref{notation}, before defining the notion of a `stick-breaking' measure and related ingredients in Section \ref{stick-breaking-definitions}.  In Section \ref{clumped}, we discuss the notion of a `clumped' representation of the stick-breaking measure which will be useful in the later discussion of protein interactions.

\subsection{Notation on spaces and conventions}
\label{notation}

We will concentrate on discrete spaces $\X$, finite or countably infinite.  Denote the space of probability measures on $\X$ by 
$$\Delta_\X=\left\{(p_i)_{i\in\X}\in[0,1]^\X:\sum_{i\in\X}p_i=1\right\}.$$

Define also that a {\it generator matrix} $G$ on $\X$ is the square matrix or operator $G=(G_{i,j})_{i,j\in\X}$ such that $G_{i,j}\geq 0$ when $i\neq j$ and $\sum_j G_{i,j} = 0$ for each $i\in \X$.  If the entries of $G$ are bounded, we say $G$ is a {\it bounded} generator matrix.  We say that $G$ is an {\it irreducible} generator matrix when for each pair $(i,j)\in \X^2$ there is a power $k=k_{i,j}$ such that $(G^k)_{i,j}>0$. 
We say $G$ has a {\it stationary} distribution $\mu\in \Delta_\X$ when $\mu$ is a left eigenvector with eigenvalue $0$, that is
$\sum_i \mu_iG_{i,j} =0$ for all $j\in \X$.  When $G$ is irreducible and has a stationary distribution $\mu$, then $\mu$ is unique.  We observe that on a finite state space $\X$, $G$ is bounded, and when $G$ is irreducible, it has a unique stationary distribution $\mu$.

We remark that a bounded generator matrix can always be (non-uniquely) decomposed as $\theta(Q-I)$ where $\theta>0$ and $Q$ is a stochastic matrix or operator. When $G$ is irreducible, then $Q$ is irreducible and additionally $G$ and $Q$ have the same stationary probability vector(s) $\mu$ (independent of the choice of $\theta$).

We now enumerate several conventions used throughout the article.

\begin{enumerate}[$\bullet$]
	\item If $v=\mathbb{R}^\X$, then $D(v)$ denotes a square diagonal matrix or operator over $\X$ whose $i$th entry is $v_i$ for each $i\in\X$. If $A\subset\X$, then $D(A)=D(v)$ where $v=\sum_{i\in A}e_i$ where $\{e_i\}_{i\in \X}$ is the standard basis of $\mathbb{R}^\X$.
	\item $\N=\{1,2,3,...\}$ and $\N_0=\{0,1,2,...\}$
	\item We define empty sums $\sum_\emptyset=0$, empty scalar products $\prod_\emptyset=1$, and empty matrix products as the identity $\prod_\emptyset=I$.
	\item Products:  For a collection of matrices $\{M_j\}_{j=1}^k$, we denote the standard forward order product as $\prod_{j=1}^k M_j=M_1\cdot M_2\cdots M_k$ and the non-standard reverse order product as $\prod_{j=1}^{k;(R)}M_j=M_k\cdot M_{k-1}\cdots M_1$.
	\item Adjoints:  
	Given a probability vector $\mu$ over $\X$, we define the adjoint $A^*$ of a square matrix or operator $A$ on $\X$ with respect to $\mu$ by $A^*=D(\mu)^{-1}A^TD(\mu)$. For a generator $G$ with $G=\theta(Q-I)$ having unique stationary distribution,  we always understand $G^*$ and $Q^*$ to be adjoints taken with respect to the associated stationary distribution.

\end{enumerate}

\subsection{Stick-breaking measures}
\label{stick-breaking-definitions}

Before describing a generalization of the Dirichlet process with respect to $\theta>0$ and a probability vector $\mu$ on $\X$, which will form the backbone of our work, we first define basic notions.  The classical Dirichlet process, much used in Bayesian nonparametric statistics, is a distribution on the space of probability measures on $\X$ with the property that a sample measure $D$ is such that the joint distribution of $\big(D(A_1),\ldots, D(A_k)\big)$ is that of a Dirichlet distribution with parameters $\big(\theta\mu(A_1),\ldots, \theta\mu(A_k)\big)$ for finite partitions $\{A_i\}_{i=1}^k$ of $\X$.  

Such a process admits a `stick-breaking' representation involving two ingredients:  a GEM residual allocation model as well as an independent sequence of i.i.d. random variables $\{T_i\}_{i\geq 1}$ on $\X$ with common distribution $\mu$.  See \cite{Ghosal_VanVaart,Muller} for more on stick-breaking measures. The GEM model is defined as follows.

\begin{defi}[GEM residual allocation model]
	Let $(Z_j)_{j\ge1}$ be an iid sequence of \emph{Beta}$(1,\theta)$ variables, and define
	$$P_j=Z_j\prod_{i=1}^{j-1}(1-Z_i).$$
	Then, $\P=(P_1,P_2,...)$ is said to have \emph{GEM}$(\theta)$ distribution.
\end{defi}

Define now the (random) `stick-breaking' measure on $\X$,
$$\nu = \sum_{j\geq 1} P_j \delta_{T_j}.$$
It is well-known that the law of $\nu$ is that of the Dirichlet process on $\X$ with parameters $(\theta, \mu)$.

We now consider a generalization where $\{T_i\}_{i\geq 1}$ is a a stationary Markov chain on $\X$ with stationary distribution $\mu$.  Such a generalization was first considered in \cite{DLS} in the context of empirical distribution limits of `simulated annealing' time-inhomogeneous Markov chains.

\begin{defi}[MSBM$(G)$, MSBMI$(G)$]\label{defnmsbm}
	Let $G$ be an irreducible, bounded generator matrix over $\X$, with a unique stationary distribution $\mu$ and with decomposition $G=\theta(Q-I)$. Let $\P\sim$ \emph{GEM}$(\theta)$ and let $\T$ be a stationary homogeneous Markov chain independent of $\P$ and having kernel $Q$ with stationary distribution $\mu$. Then, the random measure
	\begin{equation}\label{eq:MSBM}
	\nu_G=\sum_{j\ge1}P_j\delta_{T_j}
	\end{equation}
	taking values in $\Delta_\X$ is said to have distribution \emph{MSBM}$(G)$. Here, MSBM stands for Markovian stick-breaking measure. The pair $(T_1,\nu_G)$ is said to have \emph{MSBMI}$(G)$ distribution (MSBM and Initial).  Note that here $T_1$ is distributed according to $\mu$.
\end{defi}

	The construction of the `stick-breaking' object with \emph{MSBM}$(G)$ distribution given in the above definition is many to one due to the choice of decomposition $G=\theta(Q-I)$, though the distribution itself is independent of this choice. Valid choices of decomposition are indexed by the selection of $\theta$, where valid choices of $\theta$ fall in the interval $[\theta(G),\infty)$ where $\theta(G)=\sup_{i\in\X}|G_{i,i}|$. The series in the stick-breaking construction has the fastest rate of convergence when $\theta= \sup_{i\in \X}|G_{i,i}|$ is smallest. 
	
		We remark exactly in the situation when $G$ permits a decomposition $G=\theta(Q-I)$ such that $Q$ is constant stochastic with rows $\mu$, we determine that \emph{MSBM}$(G)=$ \emph{Dirichlet}$(\theta,\mu)$.  In this way, since $\{T_i\}_{i\geq 1}$ is i.i.d. exactly when $Q$ is constant stochastic, the MSBM measures generalize the Dirichlet process. 
	See \cite{DLS} for more discussion.
	
	Moreover, we note that the stick-breaking construction allows to bound the error in truncating the series.  This will be useful for later inference.  Indeed, for $k\geq 0$, $\sum_{j\geq k+1} P_j \delta_{T_j} \leq \sum_{j\geq k+1} P_j = \prod_{j=1}^k (1-X_j)$.  Since $-\log(1-X_j) \stackrel{d}{=}{\rm Exp}(\theta)$, we have that
	$-\log \prod_{j=1}^k(1-X_j) \stackrel{d}{=}Y_k:={\rm Gamma}(k,\theta)$. Then, the chance the error is greater than $\lambda$ is 
	\begin{align}
	\label{truncation error}
	P(\exp(-Y_k)\geq \lambda) = P(Z_\lambda \geq k)
	\end{align}
	where $Z_\lambda \stackrel{d}{=}{\rm Poisson}(-\theta\log(\lambda))$.

\subsection{Clumped stick-breaking constructions}
\label{clumped}
It will be useful to describe `clumped' representations of the MSBM$(G)$ stick-breaking measure, later useful in discussion of protein interactions.  
Let $G$ be an irreducible bounded generator matrix with stationary distribution $\mu$. Define $\theta(G)=\max_i|G_{i,i}|$.
For each $\theta\geq\theta(G)$, let $\bf P^\theta$ have GEM$(\theta)$ distribution and let $\bf T^\theta$ independent of $\bf P^\theta$ be a stationary Markov chain with transition kernel $Q^\theta=I+G/\theta$. 

Define
$$\nu^\theta(\ \cdot\ )=\sum_{j=1}^\infty P_j^\theta\delta_{T_j^\theta}(\ \cdot\ )$$
Each $\nu^\theta$ is a stick-breaking representation of MSBM$(G)$:
$\nu^\theta\ \stackrel d=\ \nu^{\theta(G)}$ for all $\theta\geq \theta(G)$.
Here, $\bf T^\theta$ is a Markov chain which may repeat, that is it may be that
$\p{T_j^\theta=T_{j+1}^\theta}>0$.

We now recall a `clumped' stick-breaking construction using the Markov chain $\bf S$ whose law corresponds to the ${\bf T^\theta}$ non-repeating transitions (cf. \cite{DLS} for more discussion).  Let $\bf S$ be a homogeneous Markov chain with initial distribution $\mu$ and transition kernel 
$$K_{i,j}=\frac{G_{i,j}}{-G_{i,i}}\mathbbm 1(i\neq j)$$

Next let $\bf Y$ be a random sequence such that $\bf Y|S$ is an independent sequence of $\{{\rm Beta}(1,-G_{S_j,S_j})\}_{j\geq 1}$ variables. Define $\bf R$ from $\bf Y$ as a residual allocation model
$$R_j=Y_j\prod_{i=1}^{j-1}(1-Y_i).$$
Form the associated stick-breaking measure
$$\nu(\ \cdot\ )=\sum_{j=1}^\infty R_j\delta_{S_j}(\ \cdot\ ).$$
Then,
$\nu\ \stackrel d=\ \nu^\theta$, and moreover we have the following `clumped' statement.

\begin{prop}[cf. Theorem 2.13 \cite{DLS}]\label{propaltMSBM}
Let $G$ be an irreducible, bounded generator matrix over $\X$ with unique stationary distribution $\mu$. Define stochastic kernel
$$K_{i,j}=\frac{G_{i,j}}{-G_{i,i}}\mathbbm 1(i\neq j).$$ Let $\bf T$ be a homogeneous Markov chain with transition kernel $K$ and initial distribution $\mu$. Let $\bf Z$ be a random sequence of [0,1]-valued random variables such that given $\bf T$, $\bf Z$ is an independent sequence with $Z_j\sim$Beta$(1,-G_{T_j,T_j})$. Form the residual allocation model ${\bf R} = \{Z_j\prod_{i=1}^{j-1}(1-Z_i)\}_{j\geq 1}$.  Then,
$$\left(T_1,\ \sum_{j=1}^\infty R_j\delta_{T_j}(\ \cdot\ )
\right)\ \sim\ \text{MSBMI}(G)$$
\end{prop}

\section{Time-inhomogeneous MCs, PDMPs, and multistate mRNA promoter processes}
\label{main_sect}

We consider now seemingly unrelated processes, which however in combination bear upon the multistate mRNA promoter process.  In the main section, we deduce results on the associated stationary distribution and in Section \ref{moments_subsect} on its moments.  We also discuss identifiability of parameters with respect to stationary mRNA levels in Section \ref{identifiability_sect}.

The first process is a time-inhomogeneous Markov chain, considered in \cite{DLS}, \cite{DS} with respect to certain `simulated annealing' models.

\begin{defi}[Inhomogeneous Chain $\M=(M_n)_{n\ge1}$ (cf. \cite{DLS}]
	Let $G$ be a bounded, irreducible generator matrix on a discrete space $\X$. We associate to $G$ the discrete time Markov chain $\M=(M_n)_{n\ge1}$ with state space $\X$ having transition kernels
	$$K_n=I+\frac Gn\mathbbm1(n>N)$$
	for sufficiently large $N$ that $K_n$ is stochastic. We denote the empirical measure of $\M$ up to time $n$ by
	$$\nu^n=\frac1n\sum_{j=1}^n\delta_{M_j}.$$
\end{defi}

In words, the Markov chain $\M$ stays on the state it is at with larger probability as $n$ grows, and switches states with probability of order $O(1/n)$.  In this way, the states in $\X$ can be considered `valleys' from which it becomes more difficult to leave as time increases.  Nevertheless, there will be an infinite number of switches in the chain.

The second process is a type of piecewise deterministic Markov process (PDMP)--informally, a pair $(X(t),E(t))$ such that $E(t)$ is a Markov jump process on $\X$ and, if $\{t_n\}_{n\geq 1}$ are the jump times of $E(t)$, then $X(t)$ evolves deterministically on each interval $[t_n,t_{n+1})$ in a manner determined by $E(t_n)$. Such a process is determined by the jump rates of $E(t)$, the transition measure of $(X(t),E(t))$, and the flows governing the deterministic behavior of $X(t)$ between jumps.  See \cite{Davis} for a more general and precise definition.

\begin{defi}[Exponential Zig-zag Process (cf. \cite{BC})]
	An exponential zig-zag process is a PDMP $(E(t), X(t))$ taking values in $\X\times \Delta_\X$ with infinitesimal generator
	$$\mathcal L_Zf(i,x)=(e_i-x)\cdot\triangledown_xf(i,x)+\sum_{j\neq i}G_{i,j}[f(j,x)-f(i,x)]=(e_i-x)\cdot\triangledown_xf(i,x)+\sum_jG_{i,j}f(j,x)$$
	where $G$ is an irreducible generator matrix on a finite space $\X$.  Such a process has a unique stationary distribution (cf. Section 3 \cite{BC}). 
\end{defi}

In words, the $E$ process switches according to rates $G$.  However, depending on the current state $E=i$, the $X_j$ values decrease at rate proportional to $X_j$ for $j\neq i$ and $X_i$ increases at rate $1-X_i$.

We now state carefully the multistate mRNA promoter process.

\begin{defi}[Multistate promoter process (cf. \cite{Herbach})]\label{defnmultistate}
	Let $G$ be an irreducible generator matrix on a finite space $\X$. Consider the jump Markov process $(E(t),M(t))$ taking values in $\X\times\N_0$ with transition rates
	$$(i,m)\rightarrow(j,n)\ \text{ at rate }\ \left\{\begin{array}{ccc}
	G_{i,j} & & n=m\\
	\beta_i & & i=j,\ n=m+1\\
	\delta m & & i=j,\ n=m-1\\
	0 & & \text{otherwise}
	\end{array}\right.$$
	for $i,\ j\in\X$ and $m,\ n\in\N_0$. We denote the stationary distribution of this process
	as $\pi_1(i,m|G,\beta,\delta)$.
	
	We also associate to the multistate promoter process a process $X(t)$ taking values in $\Delta_\X$ which is a solution to 
	$$\frac d{dt}X_i(t)=\delta\big[\mathbbm 1(E(t)=i)-X_i(t)\big].$$
It is known that the joint process $(E(t), M(t), X(t))$ has a unique stationary distribution (cf. Corollary 3.5 \cite{Herbach}). In particular, we denote the stationary distribution of $(E(t),M(t))$ by $\pi_1(i,m|G,\beta,\delta)$.
	
\end{defi}

	The multistate promoter process models mRNA production by a gene promoter which can be in one of a finite collection $\X$ of states. The Markov jump process $E(t)$ with rates $G$ tracks the state of the promoter over time.   The rate of mRNA production while the promoter is in state $i$ is given by $\beta_i>0$. Then, the production of mRNA is a birth-death process, with mRNA produced at rate $\beta_i$ when $E(t)=i$, while each individual mRNA degrades independently at rate $\delta>0$.

\medskip
We now state three results on the these processes and deduce the stick-breaking representation of the multistate mRNA promoter process in Theorem \ref{cormultistatestationaryconstruction}.

The first result is that the empirical measure of the time-inhomogeneous MC converges weakly to the MSBM stick-breaking measure.  A different characterization for types of $G$ may also be found in \cite{DS}.

\begin{theorem}[cf. Theorem 2.13 \cite{DLS}]
	Let $G$ be a bounded, irreducible generator matrix, with stationary distribution $\mu$, over a discrete space $\X$. Let $\M$ be the inhomogeneous chain associated to $G$, and $(\nu^n)_{n\ge1}$ be the empirical measures of $\M$. Then
	$$(M_n,\nu^n)\xrightarrow{\ d\ }\text{\emph{MSBMI}}(G^*).$$
\end{theorem}

The second result is that the stationary distribution of the PMDP is the limit empirical measure for the time-inhomogeneous MC.  In \cite{BC}, one may also find a non stick-breaking characterization of the limit, as well as other interesting results.

\begin{theorem}[cf. Theorem 2.8 \cite{BC}]
Let $G$ be an irreducible generator matrix over a finite space $\X$. Let $\M$ be the inhomogeneous chain assocaited to $G$, and $(\nu^n)_{n\ge1}$ be the empirical measures of $\M$. Let also $(E(t),X(t))$ be an exponential zig-zag process parametrized by $G$. Then, the associated stationary distribution $(E,X)$ is the limit distribution of the time-inhomogeneous Markov chain:
$$(E,X)
\stackrel d= \lim_{n\rightarrow\infty}(M_n,\nu^n).$$
\end{theorem}

The third result finds that the stationary distribution of the multistate promoter process is a certain Poisson mixture. In \cite{Herbach}, generating functions of the stationary distribution are also given.

\begin{theorem}[cf. Proposition 4.1 \cite{Herbach}]\label{marginalmultistate}
	Let $G$ be an irreducible generator matrix of a finite space $\X$. Let $\beta\in(\mathbb R^+)^\X$ and $\delta,\ \lambda>0$. Let $(E(t),M(t))$ be a multistate promoter process parametrized by $G$ with associated process $X(t)$. Suppose $M(0)|E(0),X(0)\sim$ \emph{Poisson}$(\lambda(0))$ where $\lambda(t)$ satisfies 
	$$\lambda(t)=\delta^{-1}\beta\cdot X(t).$$
	Then,
	$$M(t)\Big|\big(E(\tau))_{\tau\ge 0}\ \sim\ \text{\emph{Poisson}}(\lambda(t))\hspace{1cm} \text{where}\hspace{1cm}\partial_t\lambda(t)=\beta_{E(t)}-\delta\lambda(t).$$
	Further,
	if $(E,M,X)$ is an observation from the stationary distribution of $(E(t), M(t), X(t))$, then
	$$M\Big| E,X\ \sim\ \text{\emph{Poisson}}(\delta^{-1}\beta\cdot X).$$
	
\end{theorem}

\medskip

We remark that the stationary distribution of $(E(t),M(t))$ does not depend on the initial distribution of $M(0)$, and indeed, the representation of $M(t)$ as a Poisson mixture is retained after a finite, random burnoff period; see \cite{Lin-Buchler}.

We now combine the three previous results to find a stick-breaking representation of the stationary distribution $\pi_1(i,m|G,\beta,\delta)$, that is of the limit $(E,M,X)$.

First, by scaling time by $\delta$, we see that the stationary limit of the multistate process components $(E,X)$ is the limit of the PDMP in the work of \cite{BC} with generator $G/\delta$.  In turn, the work of \cite{DLS} shows that this limit is MSBMI$(G^*/\delta)$.  Hence, we obtain the main statement of this section, namely the following theorem.

\begin{theorem}\label{cormultistatestationaryconstruction}
	Let $G$ be an irreducible generator matrix over a finite space $\X$. Let $\beta\in(\mathbb R^+)^\X$ and $\delta>0$. Let $(E(t),M(t))$ be a multistate promoter process parametrized by $G$ with associated process $X(t)$. Then,
	$$\big(E(t),M(t),X(t)\big)\xrightarrow d (E,M,X)$$
	where 
	$$(E,X)\sim\text{\emph{MSBMI}}(G^*/\delta)\hspace{1cm}\text{ and }\hspace{1cm}M\Big| E,X\ \sim\ \text{\emph{Poisson}}(\delta^{-1}\beta\cdot X).$$
	Hence, the stationary distribution $\pi_1(i,m|G,\beta, \delta)$ of $(E,M)$ is the law of the mixture 
	${\rm Poisson}(\delta^{-1}\beta \cdot X)$.
\end{theorem}

\subsection{Moments with respect to the multistate mRNA promoter process}
\label{moments_subsect}

Using the stick-breaking apparatus we may identify moments of interest.  We first state a result found in \cite{LS}.

\begin{prop}[cf. Theorem 4 \cite{LS}]
	Let $(T,\nu)\sim$ \emph{MSBMI}$(G)$ for an irreducible, bounded generator matrix $G$ with stationary distribution on a discrete space $\X$. Let $n\in\N$, $\vec k\in\N_0^n$, and $(A_j)_{j=1}^n$ be disjoint collection of subsets of $\X$. Define $\mathbb{S}(\vec k)$ to be the collection of all distinct $k$-lists consisting of $k_1$ many 1's, $k_2$ many 2's, and so on to $k_n$ many $n$'s, where $k=\sum_{j=1}^nk_j$. Then,
	$$\left(\E{\prod_{j=1}^n \nu(A_j)^{k_j}\Bigg|T=x}\right)_{x\in\X}=\left(\#\mathbb{S}(\vec k)\right)^{-1}\sum_{\sigma\in\mathbb{S}(\vec k)}\left[\prod_{j=1}^{k;(R)}(I-G/j)^{-1}D(A_{\sigma_j})\right]\vec 1.$$
\end{prop}

We may rewrite the above expression in more convenient form.
\begin{cor}\label{cormomentstar}
	Consider the context of the previous proposition.
	 If $(\tilde T,\tilde\nu)\sim$ \emph{MSBMI}$(G^*)$, then
	$$\E{\prod_{j=1}^n \tilde\nu(A_j)^{k_j}\Bigg|\tilde T=x}=\mu_x^{-1}\cdot\left(\#\mathbb{S}(\vec k)\right)^{-1}\sum_{\sigma\in\mathbb{S}(\vec k)}\mu^T\left[\prod_{j=1}^kD(A_{\sigma_j})(I-G/j)^{-1}\right]e_x.$$
\end{cor}

\begin{proof} By convention, 
\begin{align*}
&\Big[\prod_{j=1}^{k;(R)}(I-G/j)^{-1}D(A_{\sigma_j})\Big]\vec 1 \\
&\ \ = D^{-1}(\mu)(I-G^T/k)^{-1}D(\mu)D(A_{\sigma_k})
\cdots D^{-1}(\mu)(I-G^T)^{-1}D(\mu)D(A_{\sigma_1})\vec 1.
\end{align*}
Since $D(\mu)$ and $D(A_{\sigma_\cdot})$ commute, the product equals
$$D^{-1}(\mu)(I-G/k)^{-1}D(A_{\sigma_k})(I-G/(k-1))^{-1}\cdots (I-G)^{-1}D(A_{\sigma_1})\mu.$$
The $x$th entry can be found by taking the transpose and multiplying by $e_x$.  Noting that $D^{-1}(\mu)e_x = \mu_x^{-1}e_x$ finishes the calculation.
\end{proof}

We observe that we can recover a formal, if not particularly useable, expression of the stationary probabilities, and also moments of the mRNA levels $M$ in stationarity; see also \cite{Herbach} for a derivation using a PDE for the generating function. 
\begin{cor}\label{thmstationarymultistate}
	Let $G$ be an irreducible generator matrix on a finite space $\X$ having unique stationary distribution $\mu$. Let $\beta\in(\mathbb R^+)^\X$ and $\delta>0$.   Then, the stationary distribution $\pi_1$ of the multistate promoter process $(E,M)$ parametrized by $G$, $\beta$, and $\delta$ is 
	given as
	\begin{align*}
		\pi_1(i,m|G,\beta,\delta) & =\pi_1(i,m|G/\delta,\beta/\delta,1)\\
		 & =\mu^T\left[\frac1{m!}\prod_{k=1}^m D(\beta/\delta)(I-G/(\delta k))^{-1}\right]\left[\sum_{n\ge0}\frac{(-1)^n}{n!}\prod_{k=m+1}^{m+n}D(\beta/\delta)(I-G/(\delta k))^{-1}\right]e_i.
	\end{align*}
	
	Further, by the factorial moment property of the Poisson distribution, for each $k\in \N_0$,
	\begin{align}
	\label{factorial_mom}
	 & \E{M(M-1)(M-2)\cdots (M-k+1)}\nonumber\\
	 & =\E{(\beta\cdot X/\delta)^k}=\mu^T\left[\prod_{j=1}^k D(\beta/\delta)(I-G/(\delta j))^{-1}\right]\vec 1.
	\end{align}
	where $X\sim$MSBM$(G^*/\delta)$.
\end{cor}

\begin{proof}
Note that the Poisson mixture relation $(T,{\rm Poisson}(\delta^{-1}\beta \cdot \nu))\sim\pi_1$ for $(T,\nu)\sim$MSBMI$(G^*/\delta)$ is stated in Theorem \ref{cormultistatestationaryconstruction}.  The stationary probability formulas now follow from the moment calculation \eqref{factorial_mom}. To verify these
observe $\beta\cdot \nu/\delta = \sum_{i\in \X}\delta^{-1}\beta_i\nu(i)$ and
$${\mathbb E} \big[(\beta \cdot X/\delta)^k\big]
 = \delta^{-k}\sum_{j_1,\ldots, j_k} \beta_{j_1} \cdots \beta_{j_k} {\mathbb E}\big[\nu(j_1) \cdots \nu(j_k)\big]$$
where $1\leq j_1,\ldots, j_k\leq |\X|$.
One can now check, via Corollary \ref{cormomentstar}, that the desired formula is obtained.
\end{proof}

\subsection{On identifiability of mRNA levels}
\label{identifiability_sect}

Consider the multistate mRNA promoter process $(E(t), M(t), X(t))$ parametrized by $G$, $\beta$ and $\delta$ in Definition \ref{defnmultistate}.  To begin the discussion, we will scale out the parameter $\delta$ and take it as $\delta=1$.  Let $(E,X)$ represent the stationary distribution of the process $(E(t),X(t))$. In \cite{Herbach}, it is shown that $(E,X)$ is identifiable by $G$, that is two different generators cannot give the same stationary distribution. 

 Indeed, we sketch the argument for the convenience of the reader:  The associated Laplace transform of $(E,X)$ is $\phi(s)= (\phi_1\ldots, \phi_{|\X|})$, where $\phi_i(s) = E[1(E=i)e^{s\cdot X}]$  satisfies
$$\sum_i s_i \partial_{s_i} \phi(x) = \big(D(s) + G^T)\phi(s).$$
Then, for fixed $s=\beta$, the Laplace transform of $(E, \beta\cdot X)$ is $\Phi(w) = \phi(w\beta)$ where
\begin{align}
\label{ident0}
w\partial_w\Phi = \big(wD(\beta) + G^T\big)\Phi.
\end{align}
In Corollary 4.3 of \cite{Herbach}, $\Phi(w)$ is developed in power series, $\Phi(w)=\sum_{k\geq 0}c_k(\beta)w^k$, where in particular the characterization $c_1(\beta) = (I-G^T)^{-1}D(\beta)\mu$ is made, where $\mu$ is the distribution of $E$. Note that $\mu=\Phi(0)=\Phi'(0)=\mu'$. Hence, if there are two generators $G$ and $H$ for which $(E,X)_G = (E,X)_H$ in law, then $c_1(s; G) = c_1(s;H)$.  Since $s$ is arbitrary and $\mu$ has full support as $G$ is irreducible, we conclude $G=H$.

However, one may ask about identifiability of the mRNA level $M$ itself.  In Theorem \ref{cormultistatestationaryconstruction}, we see that the stationary mRNA level $M$ is determined by the distribution of $\beta\cdot X$.  Since mRNA level readings are available from lab experiments, for inference purposes, it makes sense to study the identifiability of the distribution of $\beta\cdot X$ in terms of $(\beta, G)$.  Since we are not given the distribution of $E$ and $\beta$ is not arbitrary, the previous identifiability argument for $(E,X)$ is not sufficient.  Moreover, in the refractory case, when only one component $\beta_i>0$, the rest vanishing, \cite{Herbach} shows that certain eigenvalues of $G$ are identifiable, although $G$ itself cannot be determined (however, see the example below).  Of course, when $\beta$ is a vector with common entries, $\beta\cdot X = \beta_1$, certainly $X$ cannot be identified.  Nevertheless, since the support of $\beta\cdot X$ is $[\min_i \beta_i, \max_i \beta_i]$, both $\min_i \beta_i$ and $\max_i \beta_i$ are identifiable.

To further the discussion, the Laplace transform of $\beta\cdot X$ is $\vec{1}\cdot \Phi(w)$ which satisfies
$$w \vec{1}\cdot \partial_w \Phi = w\vec{1} D(\beta)\Phi + \vec{1}G^T\Phi = w\beta\cdot \Phi$$
since $\vec{1}G^T = \vec{0}$ given that $G$ is a generator matrix.  One can develop further equations by differentiating in $w$.  One also has expressions for the moments (cf. Corollary \ref{thmstationarymultistate} or \cite{Herbach}).  

Despite the nonlinearity of these relations which seem difficult to negotiate, 
we believe that identifiability of $(\beta, G)$ holds with respect to irreducible generators $G$, when the entries of $\beta$ are strictly ordered, say $\beta =(\beta_1,\ldots, \beta_{|\X|})$ where $\beta_1>\beta_2>\cdots > \beta_{|\X|}\geq 0$, but we leave this theoretical question to a future investigation.  Numerically, in this respect, however, we observe that the work in Section \ref{sec:estimation} gives positive evidence of this claim.

We finish the section with a `refractory example', studied in the numerical study Section \ref{sec:selection}, where one can actually show identifiability:

\begin{ex}
\rm
Consider a three-state refractory model where $\beta_1>0$ and $\beta_2 = \beta_3 = 0$
and $G$ has zero entries, $G_{13} = G_{31} = 0$.  In this case, we claim the model is identifiable. 
Form
$$
G = \begin{pmatrix}
-a & a & 0\\
b & -b-c & c\\
0 & d & -d\\
\end{pmatrix},
$$
where $a$, $b$, $c$, and $d$ are positive real numbers.  According to \cite{Herbach}, the identifiable parameters of $G$ are the (nonzero) eigenvalues of $-G$ and the eigenvalues of $-G_{(1)}$, where $G_{(1)}$ is a matrix obtained by removing the first row and the first column of $G$.  [Herbach considers $H=G^T$, which gives the same formulas.]

Let $\lambda_1$ and $\lambda_2$ be the two nonzero eigenvalues of $-G$. They are the zeros of the equation (in $\lambda$) 
$
\lambda^2 - (a + b + c + d) \lambda + ad + bd + ac = 0
$.
Let $\lambda_3$ and $\lambda_4$ be the eigenvalues of $-G_{(1)}$. They are the zeros of the equation (in $\lambda$)
$
\lambda^2 - (b+c+d) \lambda + bd = 0
$.
Then $a$, $b$, $c$, and $d$ can be expressed by $\lambda_i, i = 1, 2, 3, 4$ as 
\begin{align*}
&a = \lambda_1 + \lambda_2 - (\lambda_3 + \lambda_4), \ \ \
b = \lambda_3 + \lambda_4 - (\lambda_1\lambda_2 - \lambda_3\lambda_4)/a,\\
&c = (\lambda_1\lambda_2  - \lambda_3\lambda_4)/a  - \lambda_3\lambda_4/b, \ \ \
d = \lambda_3\lambda_4 / b,
\end{align*}
meaning $a$, $b$, $c$, and $d$ are identifiable. 
\end{ex}

\section{Protein production in the multistate mRNA promoter process}
\label{protein_sect}

We discuss now an extension of the multistate promoter model which incorporates protein production.
Specifically, when the multistate promoter process $(E(t),M(t))$ is in state $(i,m)$, we formulate that individual proteins are produced at rate of $\alpha m$ and the protein level $p$ degrades at rate $\gamma p$.  Such an ansatz corresponds to the idea that each individual mRNA independently produces protein at rate $\alpha>0$ and each protein degrades at rate $\gamma>0$.  A more involved model, which we leave to future consideration, would involve `feedback' between the protein levels and the promoter-mRNA dynamics  

A version of this model was indicated in \cite{Herbach}, and stationary distributions in the `refractory' case, when only one $\beta_i$ is positive, were found in \cite{Choudhary}.  
In this context, our goal will be to derive the stationary distribution in the general $(\beta, \delta, \alpha, \gamma, G)$ model through the stick-breaking apparatus.  

The strategy will be to consider
 `bounded joint multistate mRNA-protein processes' which restrict mRNA levels below a capacity level $c$.
Such bounded joint processes
have the same abstract finite-state promoter structure as the multistate mRNA promoter process, with stationary distributions given in terms of Markovian stick-breaking measures.  The idea is to take a limit now as the capacity level $c\uparrow\infty$ to recover the stationary distribution in the general unbounded model.  Importantly, clumped representations of the MSBM's for the bounded joint process will be of use in this regard.

We now define carefully the bounded joint process.
\begin{defi}[Bounded joint multistate mRNA-protein process] \label{defnproteinbounded}
A bounded joint process is the Markov jump process $(E(t),M(t),P(t))$ on $\X\times\{0,1,2,...,c\}\times\N_0$ with rates
$$(i,m,p)\rightarrow(j,n,q)\ \text{ at rate }\ \left\{\begin{array}{cclll}
G_{i,j} & ; & i\neq j, & n=m, & q=p\\
\beta_i & ; & i=j, & n=m+1\leq c, & q=p\\
\delta m & ; & i=j, & n=m-1, & q=p\\
\alpha m & ; & i=j, & n=m, & q=p+1\\
\gamma p & ; & i=j, & n=m, & q=p-1\\
0 & ; & & \text{o.w.}
\end{array}\right.$$
We associate to this process the bounded generator matrix over finite state space $\tilde\X=\X\times\{0,1,2,...,c\}$
\begin{align}
\label{G^cformula}
	\tilde G^c_{(i,m),(j,n)} & =\mathbbm1(i\neq j,m=n)G_{i,j}+\mathbbm1(i=j,n=m+1\le c)\beta_i+\mathbbm1(i=j,n=m-1)\delta m\nonumber\\
	 & \hspace{2cm}+\mathbbm1(i=j,m=n)(G_{i,i}-\beta_i\mathbbm 1(m<c)-\delta m)
\end{align}
and denote its stationary distribution
as $\pi_2^c(i,m,p|G,\beta,\delta,\alpha,\gamma,c)$.
\end{defi}

One may understand the bounded joint process as follows.
Let $(E(t),M(t),P(t))$ be a bounded joint process with generator $G$, production rates $\beta$ and $\alpha$, death rates $\delta$ and $\gamma$, and cap $c$. Denote the same process as 
\begin{align}
\label{bdd_relation}
(\tilde E(t),\tilde M(t)) \ {\rm with \ }\tilde E(t)=(E(t),M(t)) \  {\rm and \ }\tilde M(t)=P(t).
\end{align} 

Then, we observe $(\tilde E(t),\tilde M(t))$ is a multistate promoter process taking values in $\tilde\X\times\N_0$ parameterized by generator $\tilde G^c$, production rates $\tilde\beta_{i,m}=\alpha m$, and death rate $\tilde\delta=\gamma$. In particular, as $\tilde \X$ is a finite space,
for all $(i,m,p)\in\X\times\{0,1,\ldots, c\}\times \N_0$,
\begin{align}
\label{propequatemodels}
\pi_2^c(i,m,p|G,\beta,\delta,\alpha,\gamma,c)=\pi_1((i,m),p|\tilde G^c,\tilde\beta,\tilde\delta),
\end{align}
for which there is a stick-breaking relation via Theorem \ref{cormultistatestationaryconstruction}.

\medskip
We now state carefully the unbounded joint process.
\begin{defi}[(Unbounded) joint multistate mRNA-protein process] The (unbounded) joint process is the Markov jump process $(E(t),M(t),P(t))$ on $\X\times\N_0^2$ with rates
$$(i,m,p)\rightarrow(j,n,q)\ \text{ at rate }\ \left\{\begin{array}{cclll}
G_{i,j} & ; & i\neq j, & n=m, & q=p\\
\beta_i & ; & i=j, & n=m+1, & q=p\\
\delta m & ; & i=j, & n=m-1, & q=p\\
\alpha m & ; & i=j, & n=m, & q=p+1\\
\gamma p & ; & i=j, & n=m, & q=p-1\\
0 & ; & & \text{o.w.}
\end{array}\right.$$
We associate to this process the unbounded generator matrix
\begin{align*}
\tilde G^\infty_{(i,m),(j,n)}&=\mathbbm1(i\neq j,m=n)G_{i,j}+\mathbbm1(i=j,n=m+1)\beta_i\\
&\ \ \ \ \ +\mathbbm1(i=j,n=m-1)\delta m+\mathbbm1(i=j,m=n)(G_{i,i}-\beta_i-\delta m)
\end{align*}
and denote its stationary distribution as $\pi_2^\infty(i,m,p|G,\beta,\delta,\alpha,\gamma)$.
\end{defi}

We remark that existence of a unique stationary distribution $\pi_2^\infty$ which integrates $e^{\epsilon_1i+\epsilon_2m+\epsilon_3 p}$, for $\epsilon_1,\epsilon_2,\epsilon_3>0$, follows from an application of Theorem 1.1 \cite{Meyn}.

The type of association with a multistate mRNA process made earlier with respect to the bounded joint process cannot be implemented directly with respect to the unbounded joint protein process. Indeed, the generator matrix $\tilde G^\infty$ associated to the unbounded joint protein process is itself unbounded, and hence cannot be normalized to construct a stochastic kernel as discussed after Definition \ref{defnmsbm}. However, we will see that the `clumped' stick-breaking construction of Proposition \ref{propaltMSBM} may still be understood and used in this context.

\begin{theorem}
\label{protein_thm}
	Let $(E(t),M(t),P(t))$ be an unbounded joint process with $E$-generator $G$, production rates $\beta$ and $\alpha$, and death rates $\delta$ and $\gamma$. Then, the associated stationary distribution $\pi_2(i,m,p|G,\beta,\delta,\alpha,\gamma)$ may be sampled as follows:
	
	Define a stochastic kernel $K$ over $\X\times\N_0$ by
	$$K_{(i,m),(j,n)}=\frac{\tilde G^{\infty,*}_{(i,m),(j,n)}}{-\tilde G^{\infty,*}_{(i,m),(i,m)}}\mathbbm 1((i,m)\neq(j,n))$$
	
	 Let now $\bf T$ be a homogeneous Markov chain over $\X\times\N_0$ with transition kernel $K$ and initial distribution $\pi_1$, which may be sampled as specified by Definition \ref{defnmsbm} and Theorem \ref{cormultistatestationaryconstruction}. Conditioned on $\bf T$, let $\bf Z$ be a sequence of independent random variables with $Z_j\sim$Beta$(1,-\tilde G^\infty_{T_j,T_j}/\gamma)$.  Consider the residual allocation model ${\bf R}=\{Z_j\prod_{i=1}^{j-1}(1-Z_i)\}_{j\geq 1}$, and define a random vector $X\in\Delta_{\X\times\N_0}$ by
	$$X_{(i,m)}=\sum_{j=1}^\infty R_j\delta_{T_j}((i,m))$$
	Then, if
	$$P\Big|{\bf T,R}\sim\text{Poisson}\left(\frac\alpha\gamma\sum_{(i,m)}mX_{(i,m)}\right)$$
	and we denote $T_1=(E,M)\sim \pi_1(i,m|G, \beta,\delta)$, the stationary distribution
	$$\pi_2(\ \cdot\ ,\ \cdot\ ,\ \cdot\ |G,\beta,\delta,\alpha,\gamma)\sim (E,M,P),$$
	 is the limit in terms of the stationary distributions $\pi_1(i,m,p|\tilde G^c, \tilde \beta^c,\tilde \delta) \sim (E^c, M^c, P^c)$ of bounded joint processes (cf. \eqref{propequatemodels}),
	$$\pi_2(i,m,p|G,\beta,\delta,\alpha,\gamma)\stackrel{d}{=} \lim_{c\rightarrow\infty}\pi_1((i,m),p|\tilde G^c,\tilde\beta^c,\tilde \delta),$$
	and the joint moments of $(M,P)$ can be captured in terms of the limit
	$${\bf E}[M^kP^\ell] = \lim_{c\rightarrow\infty} {\bf E}[(M^c)^k(P^c)^\ell]$$
	where ${\bf E}[(M^c)^k(P^c)^\ell] = {\bf E}\big[(M^c)^k{\bf E}[(P^c)^\ell| E^c, M^c]\big]$ has calculation using the relations \eqref{bdd_relation} and Corollary \ref{thmstationarymultistate}.
\end{theorem}

\section{Bayesian Inference}\label{sec:estimation}

Given the possibility to extract mRNA level readings from cells in laboratory, it is natural to explore statistical inference procedures of parameters $\beta$ and $G$ from data (cf. \cite{1Albayrak}, \cite{12Herbach et al}). In the following, we concentrate on Bayesian estimation of model parameters and selection of an appropriate underlying model based on mRNA readings. The methods are demonstrated by synthetic experiments where the data are samples from a given stationary distribution. The stick-breaking form of the stationary distribution $(E, M, X)$ in the multistate mRNA promoter model with parameters $\beta$ and $G$ (having scaled $\delta =1$), will be useful in this regard. In particular, one can directly sample from the stationary distribution to a given level of accuracy by truncating the series.
  
  To compare with literature, in \cite{1Albayrak} and \cite{12Herbach et al}, inference procedures were performed for parameters in a multistate promoter mRNA-protein interaction network where the promoter space $\X = \{0,1\}$ has two states.  As remarked in the two-state setting, the mRNA level stationary level $M$ is a Poisson-Beta mixture.  With certain approximations, extending also to the stationary protein level $P$, results were found in accord with laboratory data.  
  
  From a different point of view, in \cite{Lin-Buchler}, synthetic data taken at four time points from a multistate promoter mRNA model where $|\X| = 2, 3$ and some components of the parameters $\beta$ and $G$ vanish a priori, inference of parameters is carried out.
  
  In the following, we restrict also to $|\X|=2, 3$ state multistate promoter mRNA models. Our emphasis will be on understanding the benefit from using an explicit stick-breaking formulation of the stationary measure.  Since we also have derived a stick-breaking representation of the stationary distribution in models with protein interactions, the same formalism will apply.

We present, in Section \ref{subsec:estimation}, a Bayesian approach for estimating the parameters in the multistate promoter model based on data from the stationary distribution. Although the mass function of the stationary distribution is derived in Corollary \ref{thmstationarymultistate}, it cannot be used directly for inference due to the slow convergence of the series in its formulation. We overcome the difficulty by approximating the mass function by using Monte Carlo simulations according to the stick-breaking representation of the stationary distribution. The performance of this estimation is examined under various multistate promoter models. 

In a different track, the promoter model for a gene is often fixed a priori in the literature. The number of the promoter states and the nonzero parameters in $G$ are often assumed known before data analysis. 
In this context, 
we demonstrate in Section \ref{sec:selection} that the promoter model behind the data can be selected by utilizing the stick-breaking representation of the stationary distribution.

\subsection{Parameter estimation.}\label{subsec:estimation}
Given $L$ observations $M_1, \ldots, M_L$ from the following model
\begin{equation}\label{eq:model}
\begin{gathered}
M_l | X_l, \beta, G \overset{ind.}{\sim} \text{Poisson}(\beta \cdot X_l), l= 1, \ldots, L,\\
X_l \overset{iid}{\sim} \text{MSBM(G)}, l = 1, \ldots, L
\end{gathered}
\end{equation}
our goal is to estimate the parameters $\beta$ and $G$. We first consider the case that $\beta$ have nonzero and distinct elements and all the entries in $G$ are nonzero.
Under this setting, we describe how to estimate $\beta$ and $G$ in a Bayesian approach \cite{GelmanBayesian}. Then we discuss how to modify the procedure when a zero constraint or an equality constraint is desired for some of the parameters.

In a Bayesian framework, parameters are assumed to have a prior distribution, representing experimenter's belief on the parameter values before observing data. Given data from certain model, prior belief are updated with the information in the data to produce the posterior distribution, the conditional distribution of parameters given data. The posterior mean is a commonly used estimator for model parameters.

Without any equality or zero constraints, the parameter space of model \eqref{eq:model} has $n^2$ dimensions where $n = |\X|$. Following the discussion in Section \ref{identifiability_sect}, we assume $\beta_1 > \beta_2 > \cdots > \beta_n > 0$. 
This requirement avoids the non-identifiability issue brought by permuting the states in the multistate promoter model, but it, together with the positive constraint on the rate parameters, restricts the parameter space to a subset of the Euclidean space, which brings an extra difficulty to the statistical inference. To get rid of these restrictions, we transform the parameters $(\beta_1, \ldots, \beta_n, G_{1,2}, \ldots, G_{n,n-1})$ into 
\begin{equation}\label{eq:eta}
\eta = (\log (\beta_1 - \beta_2), \log (\beta_2 - \beta_3), \ldots, \log \beta_n, \log G_{1,2}, \ldots, \log G_{n, n-1}).
\end{equation}
As the transformation is one-to-one, estimating the original model parameters is equivalent to estimating $\eta$.

To estimate $\eta$, we consider a Bayesian approach by placing independent priors on its elements.
\begin{equation}\label{eq:model_prior}
\pi(\eta) = \prod_{j=1}^{n^2} \pi(\eta_j).
\end{equation}
We used the same log-gamma prior for $\eta_j, j=1, \ldots, n^2$ (that is $\exp(\eta_j) \overset{iid}{\sim} \text{Gamma}(0.01, 0.01)$) in all the synthetic experiments presented later in this section. Other priors such as Guassian priors can also be used.
Given the choice of prior distribution, the posterior distribution of $\eta$ is 
\begin{equation}\label{eq:posterior}
\pi(\eta \mid M_1, \ldots, M_L) = \frac{f(M_1, \ldots, M_L | \eta)\pi(\eta)}{\int f(M_1, \ldots, M_L | \eta)\pi(\eta) d\eta} \propto \prod_{l = 1}^L E[\exp(-\lambda_l)\lambda_l^{M_l}] 
\prod_{j=1} ^ {n^2} \pi(\eta_j),
\end{equation}
where $f(M_1, \ldots, M_L | \eta)$ is the probability mass function of $M_1, \ldots, M_L$, $\lambda_l = \beta \cdot X_l$, and the expectation is taken with respect to $X_l$. The Bayesian estimator is then $E(\eta | M_1, \ldots, M_L)$.

Since it is difficult to compute the posterior mean analytically, we use a Gibbs sampler \cite{GelfandSmith1990}, a special Markov Chain Monte Carlo algorithm \cite{RobertCasella2004} to draw samples from the posterior distribution. Then, the parameters are estimated by the posterior sample mean. In a Gibbs sampler, parameters are initialized at an arbitrary value $\eta^{(0)} = (\eta_1^{(0)}, \ldots, \eta_{n^2}^{(0)})$. In each iteration, each parameter is sampled from its full conditional distribution given the data and the current value of other parameters. In our case, in iteration $g$, we should draw $\eta_j^{(g)}$ from 
\begin{equation}\label{eq:full_conditionals}
\pi(\eta_j | \eta_1^{(g)}, \ldots, \eta_{j-1}^{(g)}, \eta_{j+1}^{(g-1)}, \ldots, \eta_{n^2}^{(g-1)}, M_1, \ldots, M_L).
\end{equation}

However, the distribution \eqref{eq:full_conditionals} is difficult to directly sample from. A Metropolis-Hastings (MH) algorithm \cite{Hastings1970,RobertCasella2004} is adopted to sample $\eta_j$ from \eqref{eq:full_conditionals}. 
More specifically, in iteration $g$ of the Gibbs sampler, given the current value $\eta_j^{(g-1)}$ of $\eta_j$, a proposed value $\eta_j^{\prime}$ is generated by a Gaussian random walk, that is $\eta_j^{\prime} \sim N(\eta_j^{(g-1)}, \sigma_j^2)$. Then $\eta_j^{\prime}$ is accepted as a new sample with probability
\begin{equation}\label{eq:acp_prob}
\alpha = \min\left\{1, \frac{\pi (\eta_j^{\prime} \mid \eta_1^{(g)}, \ldots, \eta_{j-1}^{(g)}, \eta_{j+1}^{(g-1)}, \ldots, \eta_{n^2}^{(g-1)}, M_1, \ldots, M_L) }{ \pi (\eta_j^{(g-1)} \mid \eta_1^{(g)}, \ldots, \eta_{j-1}^{(g)}, \eta_{j+1}^{(g-1)}, \ldots, \eta_{n^2}^{(g-1)}, M_1, \ldots, M_L)}\right\}.
\end{equation}
If $\eta_j^\prime$ is not accepted, $\eta_j^{(g-1)}$ is reused as the sample obtained in this iteration. In other words,
$$
\eta_j^{(g)} = \left\{
\begin{array}{ll}
\eta_j^\prime &\text{~with probability~} \alpha;\\
\eta_j^{(g-1)} &\text{~with probability~} 1-\alpha.
\end{array}\right.
$$

The key step of performing the MH step is to evaluate $\alpha$. As $\alpha$ is determined by the ratio of the full conditional density of $\eta_j$ at $\eta_j^{\prime}$ and $\eta_j^{(g-1)}$ and the full conditional density of $\eta_j$ is proportional to $\pi(\eta|M_1, \ldots, M_L)$, it is sufficient to evaluate the left hand side of \eqref{eq:posterior}.
Although it can not be computed exactly due to the intractable expectation, we can use Monte Carlo simulations to approximate the expectation. Given the value of $\eta$, thus $\beta$ and $G$, $E[\exp(-\lambda)\lambda^{M_l}]$ is approximated by
$$\frac{1}{B} \sum_{b = 1}^B \exp(-\lambda_b)\lambda_b^{M_l},$$
where $\lambda^{(b)} = \beta \cdot X^{(b)}$ and $X^{(b)}, b=1, \ldots, B$ are iid samples from $\text{MSBM}(G)$.

 Truncations of the stick-breaking constructions in \eqref{eq:MSBM} are used when drawing samples from $\text{MSBM}(G)$. For a given $G$, the number of terms in the truncated series can be determined explicitly based on the error tolerance. More specifically, to guarantee that the error of truncation is below $\varepsilon$ with probability higher than $1-p$, we truncated the series at term $1 + w(G, \varepsilon, p)$ where $w(G, \varepsilon, p)$ is the smallest integer $w$ such that $P(Z \leq w) \geq 1- p$ for a Poisson random variable $Z$ with parameter $-\max_{1\leq i \leq n}{|G_{i,i}|}\log(\varepsilon)$ (cf. discussion near \eqref{truncation error}). In the experiments presented later in this section, we further restrict the maximum number of terms involved in the calculations to avoid extremely long computing time for certain values of $G$.

Sometimes, it is desirable to obtain estimates of $(\beta, G)$ with zero constraints or equality constraints on the elements. For example, if one knows from previous investigation that the gene expression of interest follows a two-state refractory promoter model, then the desired estimate should have constraint $\beta_1 > \beta_2 = 0$. If it is known a priori that $X_l, l=1, \ldots, L$ in \eqref{eq:model} should follow a Dirichlet distribution in a three-state model, then one would expect an estimate of $G$ with $G_{2,1} = G_{3, 1}$, $G_{1,2} = G_{3,2}$, and $G_{1,3} = G_{2,3}$. The estimation procedure described above will not produce estimates satisfying the constraints, but a slight modification to the procedure will suffice as the constraints essentially reduce the dimension of the parameter space. Estimating constrained $(\beta, G)$ is equivalent to estimating a transformed parameter vector $\eta$ whose dimension is lower than $n^2$. In the two-state refractory promoter example, $\eta = (\log \beta_1, \log G_{1,2}, \log G_{2,1})$. In the three-state Dirichlet distribution example, $\eta = (\log (\beta_1-\beta_2), \log(\beta_2 - \beta_3), \log \beta_3, \log G_{2,1}, \log G_{1, 2}, \log G_{1,3})$. Once identifying the equivalent unconstraint parameter vector $\eta$, a Gibbs sampler similar to the one described above can be applied to estimate $\eta$. 

We conduct synthetic experiments to study the estimation performance of the above procedure. Four instances of the multistate promoter model described in \eqref{eq:model} are considered in the experiments:

\begin{itemize}
\item a two-state model,
\item a three-state model with $\text{MSBM}(G)$ being a Dirichlet distribution,
\item a three-state model with $\text{MSBM}(G)$ having a symmetric structure in $G$, and
\item a three-state model with $\text{MSBM}(G)$ having an asymmetric structure in $G$.
\end{itemize}

Three choices of sample size, $L = 100, 500, 1000$, are considered to investigate the performance of the sampler. Twenty datasets are generated for each model and each choice of $L$. The parameter estimation procedure described previously is applied to each of the datasets. The root mean squared error (RMSE) of the posterior mean estimator for each parameter is recorded for evaluation. It is computed as 
$$\text{RMSE} = \sqrt{\frac{1}{20} \sum_{i=1}^{20} (\hat\theta_i - \theta)^2},$$ where $\theta$ is parameter value used for generating data and $\hat\theta_i$ is the estimated value from the $i$th dataset. Smaller RMSE indicates better estimation performance. In the following, we present the results for each model instance.

\subsubsection{Two-state model}
The parameter setting of the two-state model is adapted from \cite{12Herbach et al}. More specifically, we set $\beta = (1000, 1)^\top$, and
$$G=\left(\begin{array}{rr}
-10& 10\\
.34&-.34
\end{array}\right)$$
when generating data. Table \ref{table:2state} presents the RMSE for parameters in the two-state model. Figure \ref{fig:boxplot_2state} gives the boxplots of the posterior means of different parameters from 20 replications. The results show that the estimation performance improves as the sample size increases. Figure \ref{fig:density_2state} displays the posterior density of the parameters for one dataset. 

\begin{table}[htb]
\centering
\caption{RMSE for parameters in the two-state model.}\label{table:2state}
\begin{tabular}{ccccc}
\hline
$L$ & $\beta_1$ & $\beta_2$ & $G_{2,1}$ & $G_{1,2}$ \\
\hline 
100 & 542.73 & 0.41 & 0.06 & 5.84 \\
500 & 412.50 & 0.13 & 0.03 & 4.36 \\
1000 & 319.84 & 0.13 & 0.02 & 3.36\\
\hline
\end{tabular}
\end{table}

\subsubsection{Three-state Dirichlet model}
The parameter values for the three-state Dirichlet model are
$\beta = (1000, 100, 1)^\top$ and
$$G = \left(
\begin{array}{rrr}
-11 & 1.0 & 10.0\\
0.34 & -10.34 & 10.0\\
0.34 & 1.0 & -1.34
\end{array}\right).
$$
Table \ref{table:3state_diri} and Figures \ref{fig:boxplot_3state_diri} and \ref{fig:density_3state_diri} present the results for the three-state Dirichlet model. Trends are similar to those observed in the two-state model.
\begin{table}[htb]
\centering
\caption{RMSE for parameters in the three-state Dirichlet model.}\label{table:3state_diri}
\begin{tabular}{ccccccc}
\hline
$L$ & $\beta_1$ & $\beta_2$ & $\beta_3$ & $G_{2,1}$ & $G_{3,2}$ & $G_{1,3}$  \\
\hline
100 & 438.89 & 118.91 & 2.13 & 0.10 & 0.56 & 3.49\\
 500& 327.74 & 112.28 & 1.17 & 0.07 & 0.35 & 3.38\\
1000& 269.92 & 63.50 & 0.78 & 0.06 & 0.23 & 2.80\\
\hline
\end{tabular}
\end{table}

\subsubsection{Symmetric three-state model}\label{sec:model3}
The parameters used in the general symmetric three-state model for generating data are $\beta = (300, 150, 20)^\top$ and
$$G = \left(
\begin{array}{rrr}
-2.0 & 2.0 & 0.0\\
0.5 & -1.0 & 0.5\\
0.0 & 2.0 & -2.0
\end{array}\right).
$$ 
The model structure of $G$ is known a priori, meaning that $G_{1,3}$ and $G_{3,1}$ are fixed at zero in all iterations of the Gibbs sampler for estimating $$\eta = (\log (\beta_1 - \beta_2), \log(\beta_2 - \beta_3), \log(\beta_3), \log(G_{1,2}), \log(G_{2,1}), \log(G_{2,3}), \log(G_{3,2})).$$ Table \ref{table:3state_gen} and Figures \ref{fig:boxplot_3state_beta}--\ref{fig:density_3state_G} present the results for estimating the nonzero parameters. In general, similar to the results for the two previous models, the estimation accuracy for all the parameters improves as the sample size increases. However, the RMSE of the parameters in $G$ for $L=500$ is greater than that for $L=100$. This happens because one of the 20 datasets produces an estimated value far from the true value. This outlier distorts the value of RMSE in the case $L=500$. As the boxplots in Figure \ref{fig:boxplot_3state_G} show, the estimates from most of the datasets do improve as the sample size increases.

We observe that the estimated values of the parameters are in close vicinity of the true values especially when the sample size is large. Although we do not theoretically prove the identifiability of the general multistate promoter models, this observation suggests that the model is likely to be identifiable at least for the three-state case considered here.

\begin{table}[htb]
\centering
\caption{RMSE for parameters in the symmetric three-state model.}\label{table:3state_gen}
\begin{tabular}{cccccccc}
\hline
$L$ & $\beta_1$ & $\beta_2$ & $\beta_3$ & $G_{1,2}$ & $G_{2,1}$ & $G_{2,3}$ & $G_{3,2}$\\
\hline
100 & 133.42 & 51.68 & 8.28 & 11.80 & 7.29 & 7.80 & 2.91\\
500 & 132.94 & 40.69 & 6.09 & 9.01 & 7.88 & 18.56 & 4.07\\
1000 & 117.02 & 10.88 & 5.11 & 4.66 & 0.59 & 0.16 & 0.35\\
\hline
\end{tabular}
\end{table}

\subsubsection{Asymmetric three-state model}
We now consider another general three-state model with $\beta = (300, 150, 20)^\top$ and
$$ G = \left(
\begin{array}{rrr}
-1.0 & 1.0 & 0.0 \\
 0.0 &  -0.5 & 0.5 \\
 1.0 & 0.5 & -1.5\\
\end{array}\right).
$$
The $G$ matrix does not have a symmetric pattern as the one in Section \ref{sec:model3}. We again assume the structure of $G$ is known when estimating the parameters. 
The results shown in Table \ref{table:3state2_gen} and Figures \ref{fig:boxplot_3state2_beta}--\ref{fig:density_3state2_G} once again suggest that the estimation performance improves as the sample size increases and that the model parameters are identifiable.

\begin{table}[htb]
\centering
\caption{RMSE for parameters in the asymmetric three-state model.}\label{table:3state2_gen}
\begin{tabular}{cccccccc}
\hline
$L$ & $\beta_1$ & $\beta_2$ & $\beta_3$ & $G_{1,2}$ & $G_{2,1}$ & $G_{2,3}$ & $G_{3,2}$\\
\hline
100 & 104.24 & 37.37 & 10.03 & 3.29 &  4.36 & 0.69 & 1.05\\
500 & 14.42 & 2.75 & 5.13 & 0.27 & 0.12 & 0.18 & 0.26\\
1000 & 7.21 & 2.17 & 5.08 & 0.15 & 0.11 & 0.12 & 0.34\\
\hline
\end{tabular}
\end{table}

\begin{figure}[H]
\includegraphics[width=\textwidth]{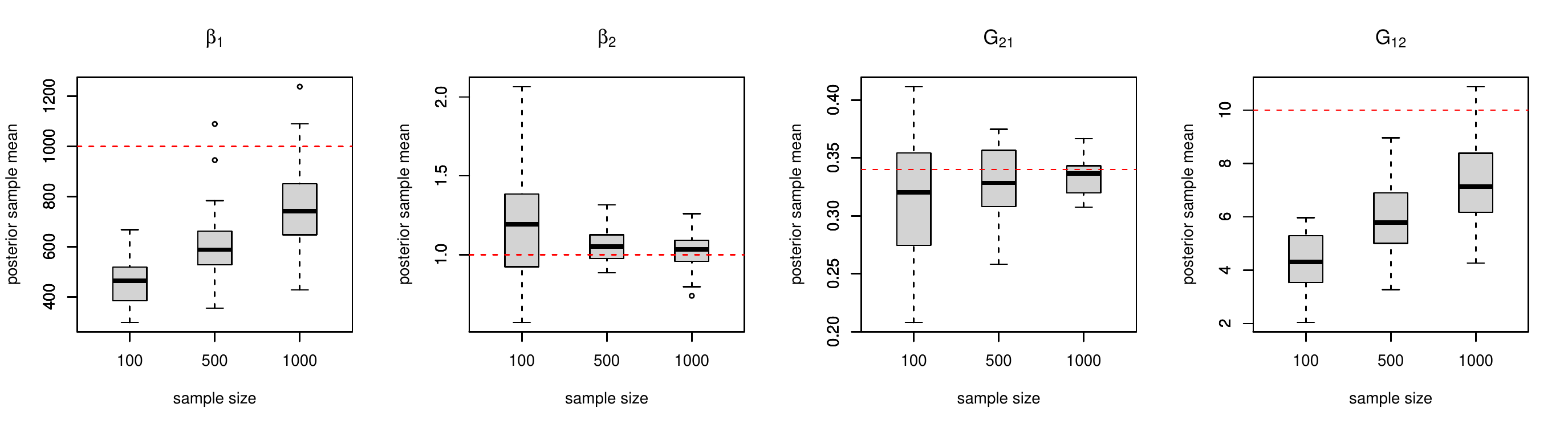}
\caption{Boxplots of the posterior means in the two-state model. The horizontal lines indicate the true values.}\label{fig:boxplot_2state}
\end{figure}

\begin{figure}[H]
\includegraphics[width=\textwidth]{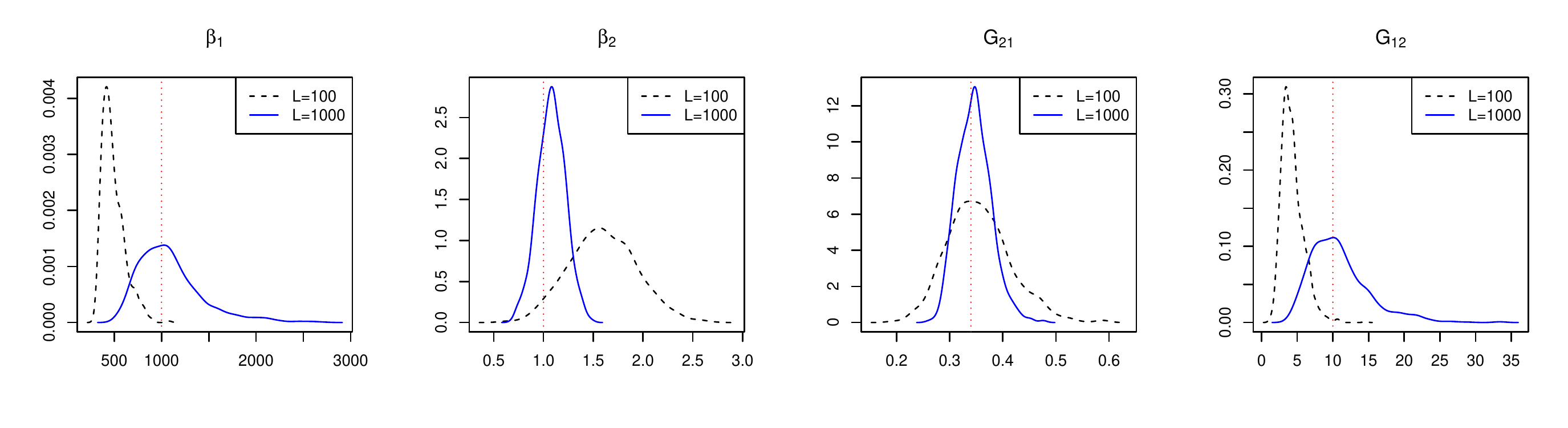}
\caption{Posterior densities of the parameters in the two-state model. The vertical lines indicate the true values.}\label{fig:density_2state}
\end{figure}

\begin{figure}[H]
\includegraphics[width=0.75\textwidth]{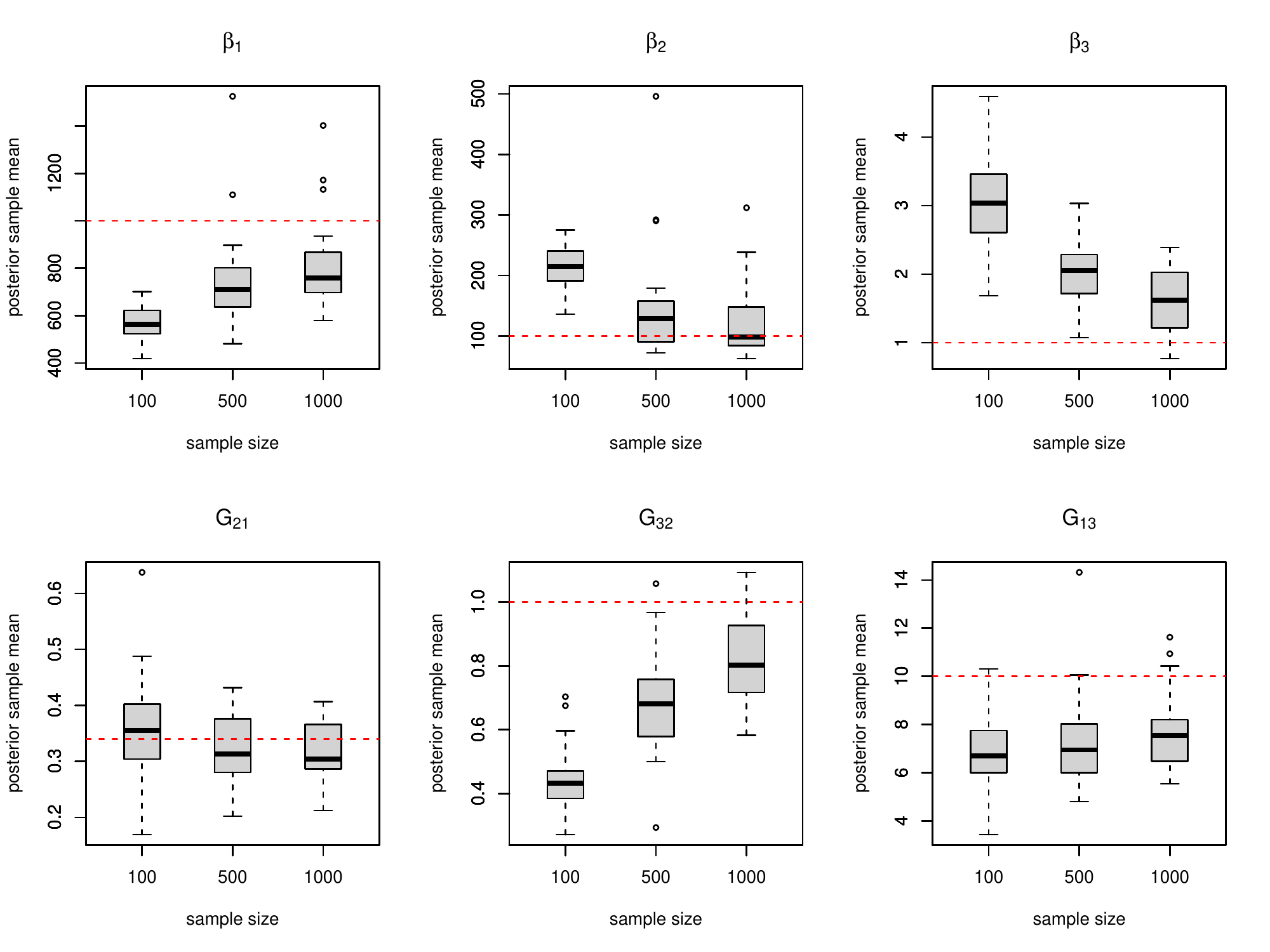}
\caption{Boxplots of the posterior means for the three-state Dirichlet model.}\label{fig:boxplot_3state_diri}
\end{figure}

\begin{figure}[H]
\includegraphics[width=0.75\textwidth]{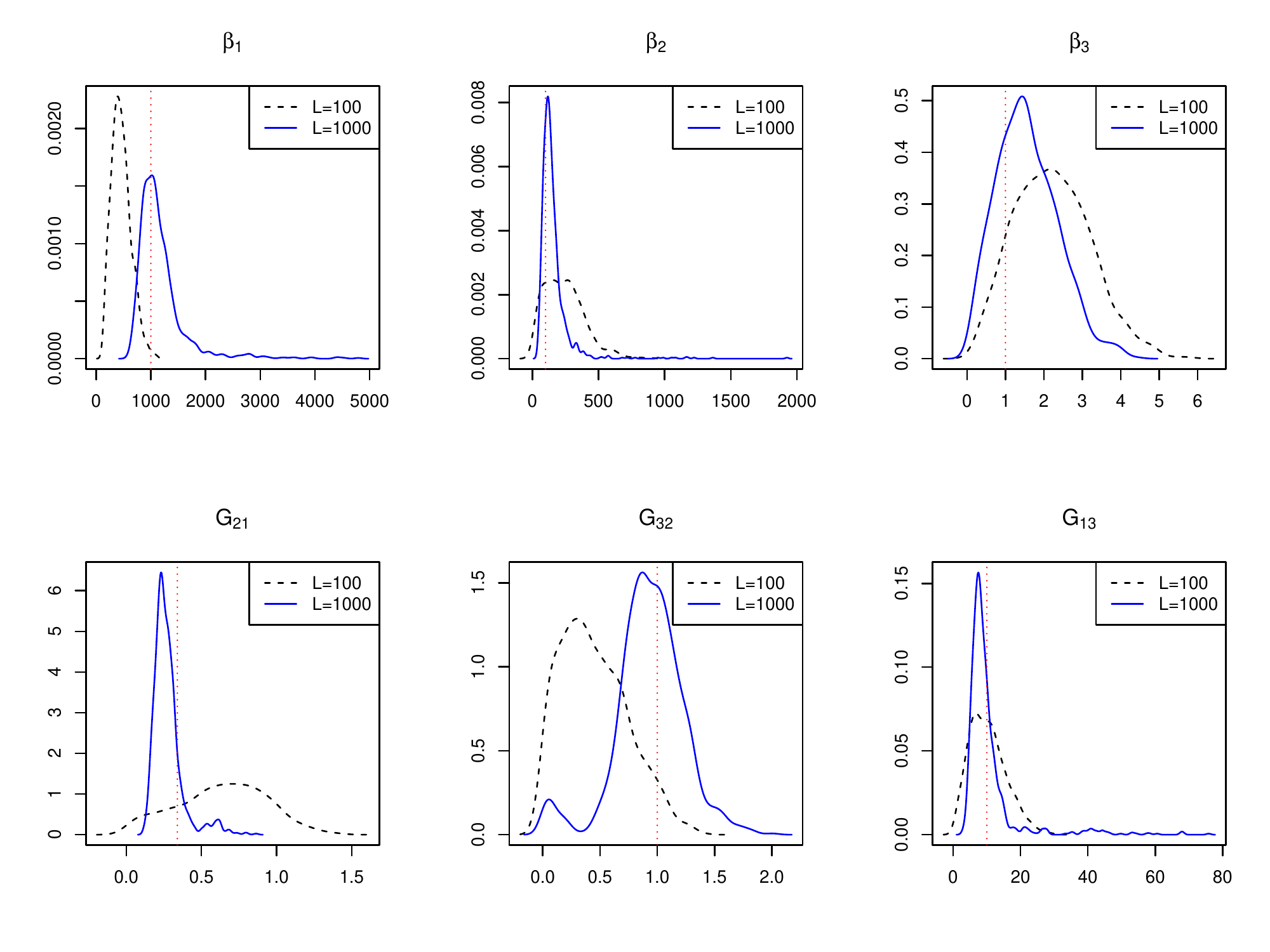}
\caption{Posterior densities of the parameters in the three-state Dirichlet model.}\label{fig:density_3state_diri}
\end{figure}

\begin{figure}[H]
\includegraphics[width=0.7\textwidth]{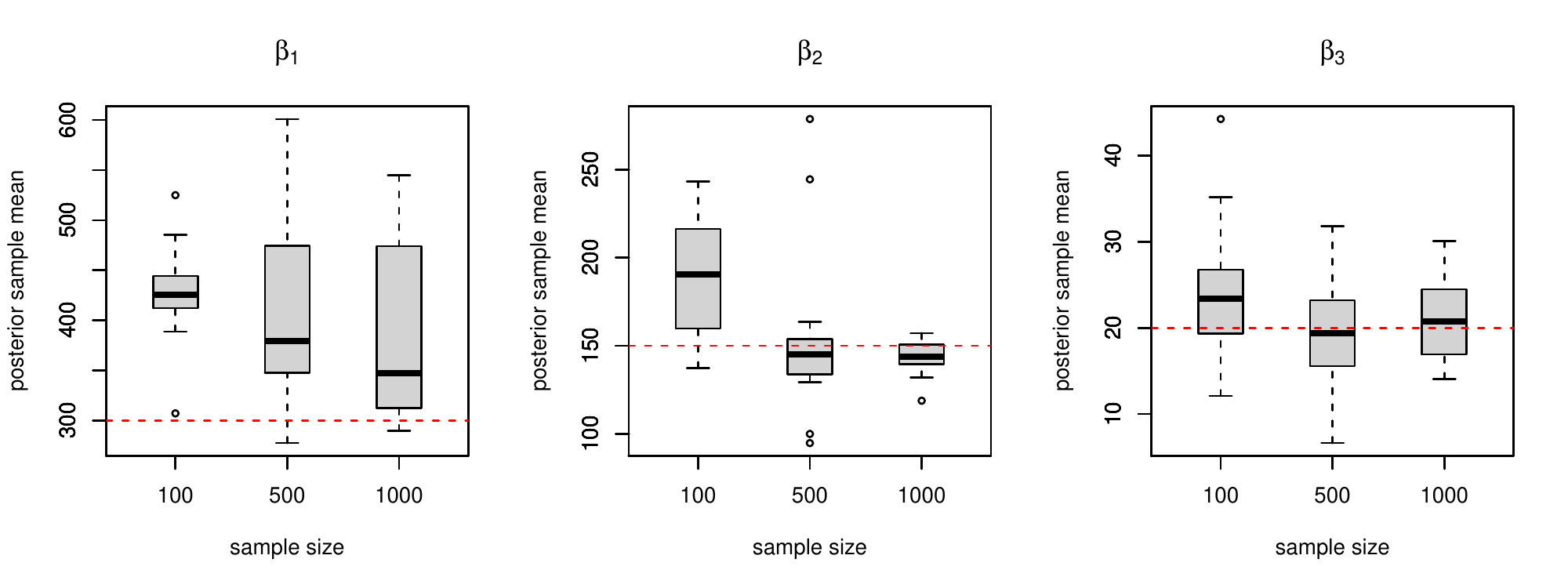}
\caption{Boxplots of the posterior means for $\beta$ in the symmetric three-state model.}\label{fig:boxplot_3state_beta}
\end{figure}

\begin{figure}[H]
\centering
\includegraphics[width=\textwidth]{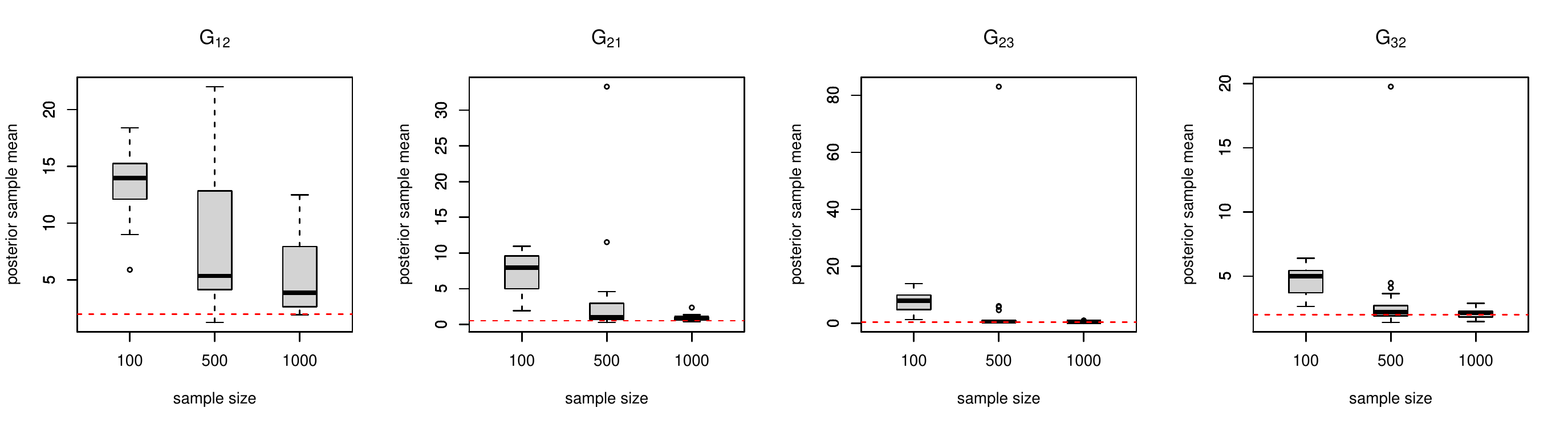}
\caption{Boxplots of the posterior means for $G$ in the symmetric three-state model.}\label{fig:boxplot_3state_G}
\end{figure}

\begin{figure}[H]
\includegraphics[width=0.75\textwidth]{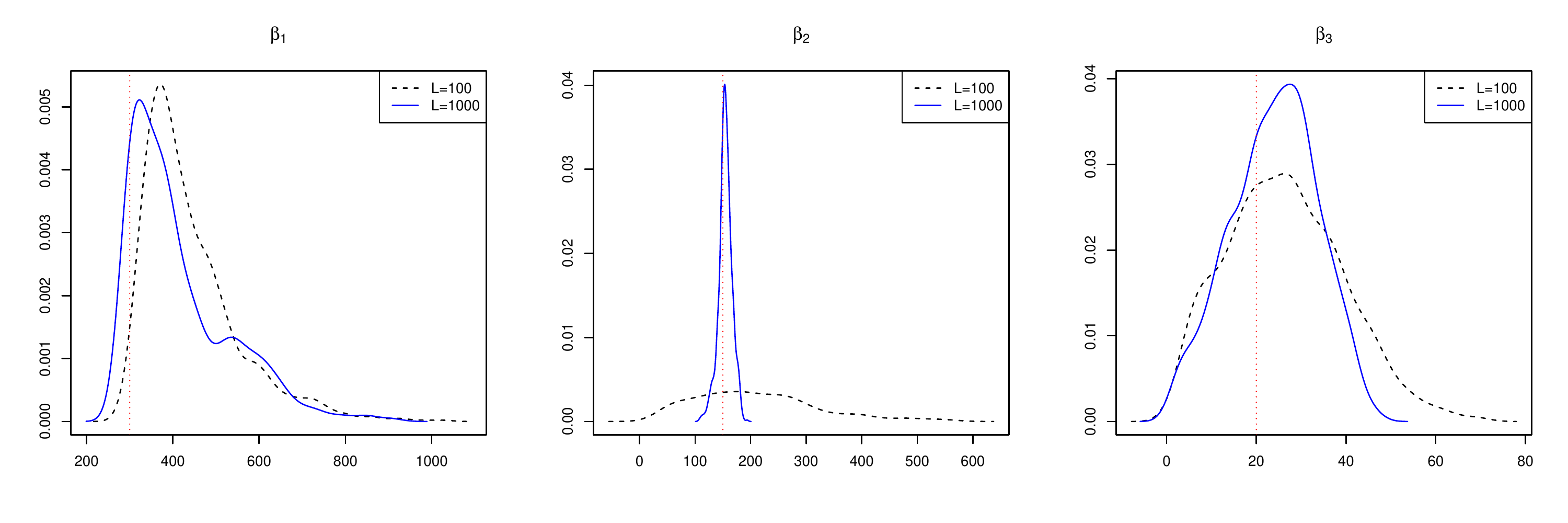}
\caption{Posterior densities of $\beta$ in the symmetric three-state model.}\label{fig:density_3state_beta}
\end{figure}

\begin{figure}[H]
\centering
\includegraphics[width=\textwidth]{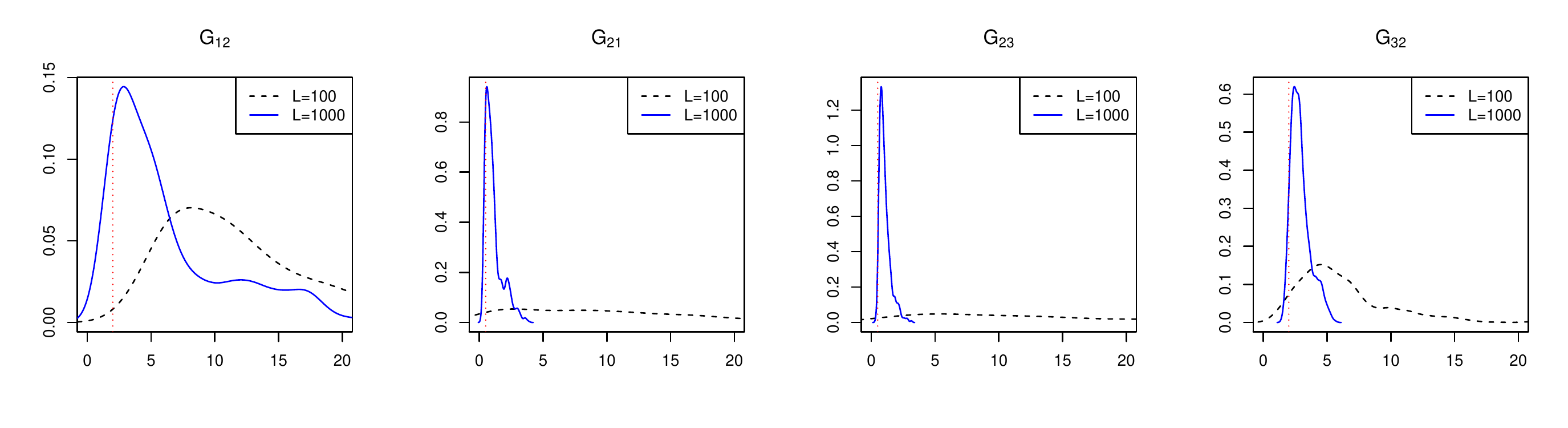}
\caption{Posterior densities of $G$ in the symmetric three-state model.}\label{fig:density_3state_G}
\end{figure}

\begin{figure}[H]
\includegraphics[width=0.75\textwidth]{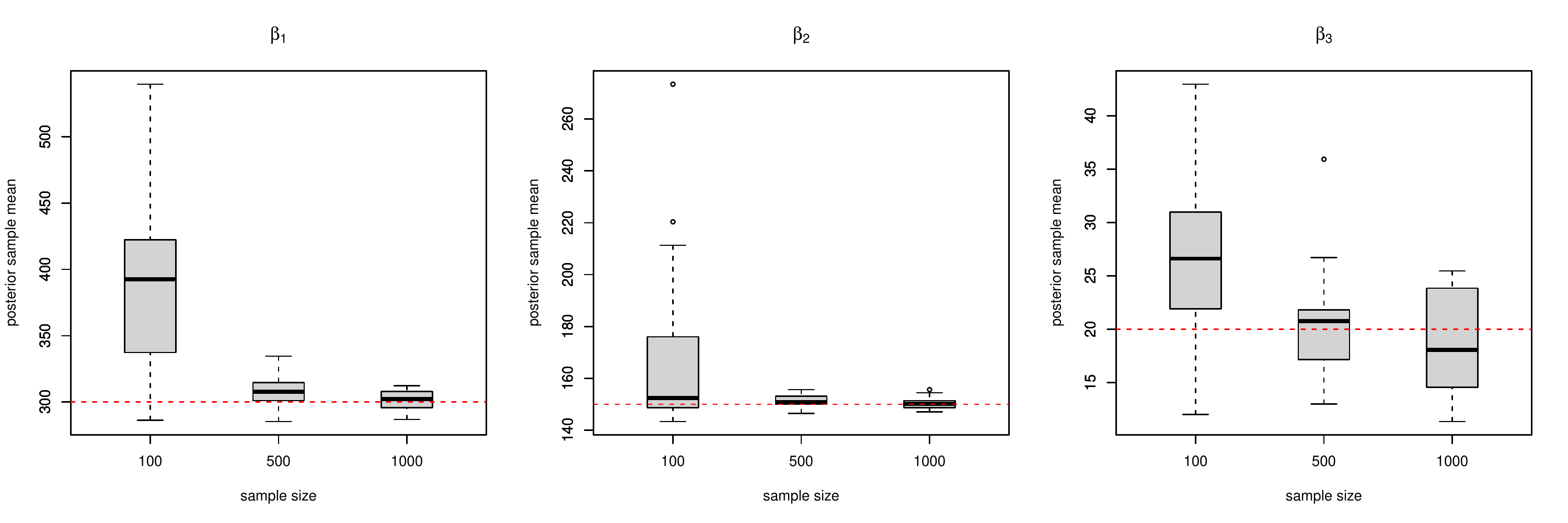}
\caption{Boxplots of the posterior means for $\beta$ in the asymmetric three-state model.}\label{fig:boxplot_3state2_beta}
\end{figure}

\begin{figure}[H]
\centering
\includegraphics[width=\textwidth]{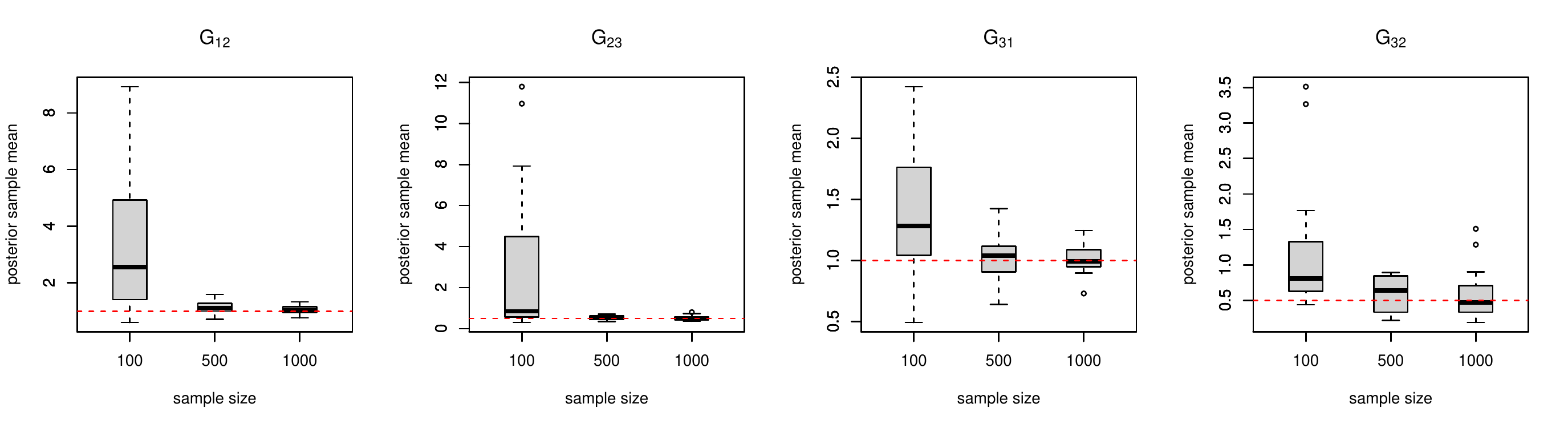}
\caption{Boxplots of the posterior means for $G$ in the asymmetric three-state model.}\label{fig:boxplot_3state2_G}
\end{figure}

\begin{figure}[H]
\includegraphics[width=0.75\textwidth]{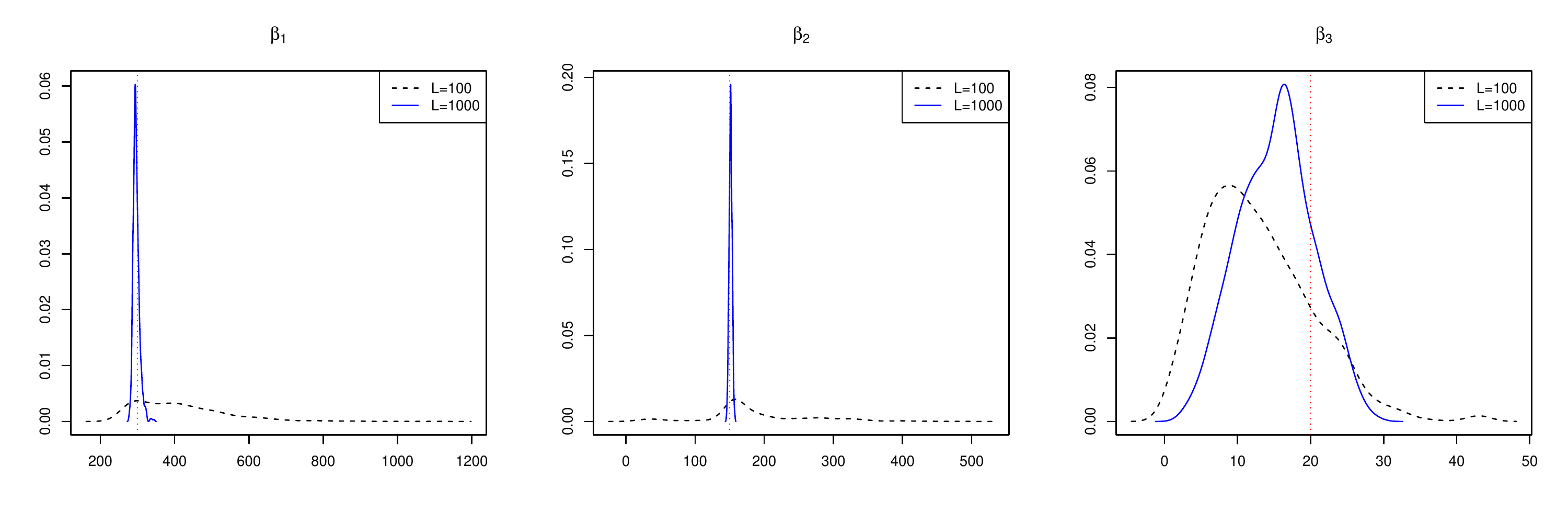}
\caption{Posterior densities of $\beta$ in the asymmetric three-state model.}\label{fig:density_3state2_beta}
\end{figure}

\begin{figure}[H]
\centering
\includegraphics[width=\textwidth]{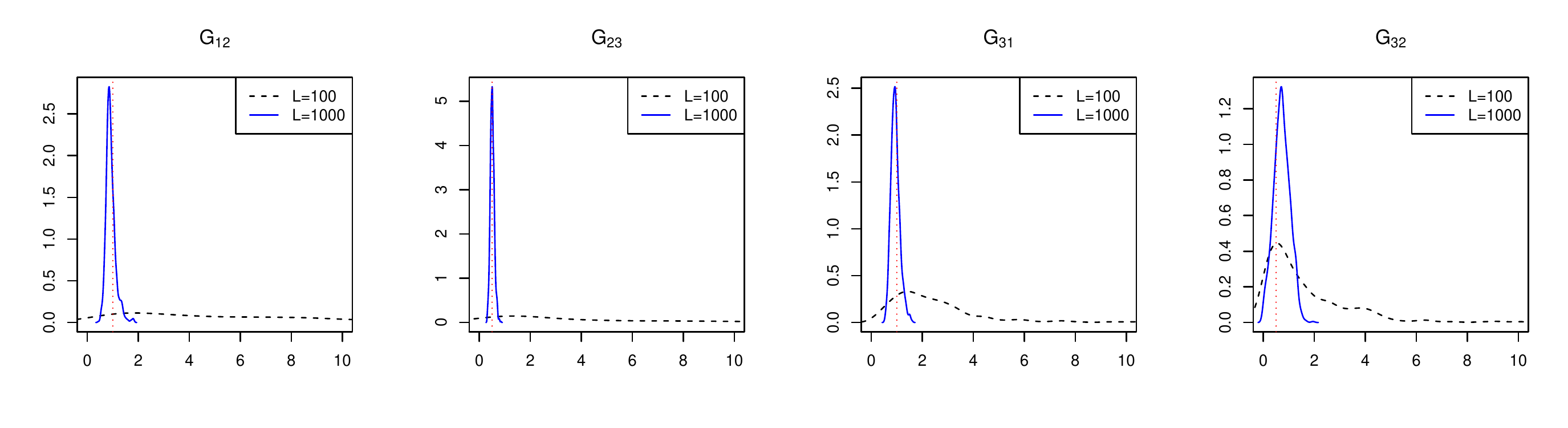}
\caption{Posterior densities of $G$ in the asymmetric three-state model.}\label{fig:density_3state2_G}
\end{figure}

\subsection{Model selection.}\label{sec:selection}
In the previous section, parameters are estimated with the assumption that the structure of the multistate promoter model (the number of states, the position of zero elements, etc.) is known. 
In this section, we consider selecting an appropriate model structure according to the Bayesian Information Criterion (BIC) \cite{Schwarz1978}. 

Suppose we want to choose from several models with different structures.
Given observed data $M_1, \ldots, M_L$, for each candidate model, the model parameter can be estimated using the procedure described in Section \ref{subsec:estimation}. Recall that $\eta$ denote the unconstraint parameter vector. Let $q$ denote the dimension of $\eta$ and $\hat\eta$ denote the estimated value of $\eta$. The BIC for a model with parameter vector $\eta$ is defined as 
$$\text{BIC} = -2 \log f(M_1, \ldots, M_L \mid \hat\eta) + \log (L) q,$$
where $f(M_1, \ldots, M_L \mid \hat\eta)$ is the probability of observing the sample under the model when $\eta = \hat \eta$. 
After computing the BIC for each candidate model, the model with the smallest BIC is chosen as the most appropriate one.

In general, a more complex model (a model with more parameters) produces a higher $f(M_1, \ldots, M_L \mid \hat\eta)$. However, an overly complex model is undesired in practice as it brings in instability in statistical inference without improving much the explanatory power. If two models give similar $f(M_1, \ldots, M_L \mid \hat\eta)$, the one with the fewer parameters is favored by BIC due to the term $\log(L) q$. Therefore, BIC can help us select the simplest model that explains the observed data reasonably well.

We demonstrate the model selection approach again using synthetic experiments. We still consider three choices of sample size $L$, 100, 500, and 1000. For each choice, 20 datasets are generated from the three-state model in Section \ref{sec:model3}.

There we estimated the nonzero parameters in $\beta$ and $G$ assuming the number of states and structures of $\beta$ and $G$ are known. Differently, in this section, we examine whether the correct underlying model (the number of states and the nonzero elements) can be identified from several candidate models using the BIC framework.

The candidate models are the following:
\begin{itemize}
\item the true model, that is the three-state model with $G_{13} = G_{31} = 0$,
\item a two-state model with all parameters being nonzero,
\item a three-state refractory promoter model with $\beta_2 = \beta_3 = G_{13} = G_{31} = 0$, and
\item a general three-state model with all parameters being nonzero. 
\end{itemize}

For each dataset, we fit all four candidate models and compute BIC for each model fit. The model with the lowest BIC is selected. Table \ref{table:3state_selection} gives the distributions of the selected model among the $20$ datasets. It shows that the model selection accuracy improves as the sample size increases. We note that the true model is selected for all $20$ datasets when $L=1000$.

\begin{table}[htb]
\centering
\caption{Distribution of the selected model among 20 replications.}\label{table:3state_selection}
\begin{tabular}{ccccc}
\hline
 & True & Refractory & Two-state & Full three-state \\
 \hline
 $L=100$ & 0 & 20 & 0 & 0\\
 $L=500$ & 14 & 0 & 6 & 0\\
 $L=1000$ & 20 & 0 & 0 & 0\\ 
 \hline
\end{tabular}
\end{table}

\section{Clumped constructions and proof of Theorem \ref{protein_thm}}
\label{clumped_sect}

Viewing the unbounded model where $(E(t), M(t))$ serves as a `promoter', define 
\begin{align*}
	\tilde e & =(e,m)\hspace{1cm}\tilde m=p\hspace{1cm}\tilde\beta_{(e,m)}=\alpha m\hspace{1cm}\tilde\delta=\gamma\\
	\\
	\tilde G^\infty_{(i,m),(j,n)} & =G_{i,j}\mathbbm 1(i\neq j,m=n)+\beta_i\mathbbm 1(i=j,n=m+1)+\delta m\mathbbm 1(i=j,n=m-1)\\
	 & \hspace{1cm}-\mathbbm 1(i=j,m=n)\left[-G_{i,i}+\beta_i+\delta m\right]
\end{align*}
As commented before Theorem \ref{protein_thm}, note that $\tilde G^\infty$ 
is not a bounded generator matrix since its diagonal entries grow unbounded with $m$, that is $\theta(\tilde G^\infty)=\infty$. As a result, $\tilde G^\infty$ cannot be normalized by some value of $\theta$ such that $I+\tilde G^\infty/\theta$ is a stochastic kernel, preventing consideration of non-clumped stick-breaking measure construction of the stationary distribution $\pi_2^\infty$ of the unbounded model parameterized by $\tilde G^\infty$ as in the discussion of the `bounded' mRNA-protein model.

We will however derive a clumped stick-breaking form of the stationary distribution $\pi_2^\infty$ of the process $(E^\infty(t), M^\infty(t), P^\infty(t))$ through a limit with respect to the `bounded' mRNA-protein model $(E^c(t), M^c(t), P^c(t))$ (c.f. Definition \ref{defnproteinbounded}).
To this end, we view the bounded mRNA-protein model and its stationary distribution $\pi_2^c$ on the full state space $(\tilde e,\tilde m)\in(\X\times\mathbb N_0)\times\mathbb N_0$, explicitly denoting dependence on the integer-valued capacity parameter $c$:
$$\tilde e  =(i,m)\hspace{1cm}\tilde m=p\hspace{1cm}\tilde\beta_{(e,m)}=\alpha m\hspace{1cm}\tilde\delta=\gamma$$
and
$$\tilde G^c_{(i,m),(j,n)}=\tilde G^\infty_{(i,m),(j,n)}\mathbbm 1(m,n\le c)$$
with the necessary modification of $\tilde G^c_{(i,c),(i,c)}$ to accord with the formula \eqref{G^cformula} and to preserve generator structure.  Note that for all $m<c$ that
\begin{align}
 \tilde G^{c,*}_{(i,m),(i,m)}=\tilde G^{\infty,*}_{(i,m),(i,m)}\label{Gsamediag}
\end{align}

By inspection of the rates, we may couple the bounded and unbounded processes so that $E^c(t)\equiv E^\infty(t)$, and also $M^c(t)\leq M^\infty(t)$ and $P^c(t)\leq P(t)$, and hence also in the $t\uparrow\infty$ limit.  Since the stationary distribution $\pi_2^\infty$ of the unbounded process integrates $e^{\epsilon_1i+\epsilon_2m+\epsilon_3p}$ for some $\epsilon_1,\epsilon_2,\epsilon_3>0$, the stationary measures $ \pi_2^c \sim (E^c, M^c, P^c)$ indexed in $c$ are tight.

We now show weak convergence of the clumped stick-breaking forms of $\pi_2^c$ to $\pi_2^\infty$.  Given uniform exponential moments, then the joint moments of $(M^c)^k(P^c)^\ell$ would converge to those of $(M^\infty)^k(P^\infty)^\ell$.  In this way, the last statement of Theorem \ref{protein_thm} would hold.

Let $\mu^c
=\pi_1^c(i,m|G,\beta,\delta)$ 
be the unique stationary measure of $(E^c(t), M^c(t))$ having support $\{(i,m):m\le c\}$. 
Let also $\mu^\infty=\pi_1(i,m|G,\beta,\delta)$ be the established stationary distribution $\pi_1$ of $(E^\infty(t), M^\infty(t))$, the usual mRNA portion of the (unbounded) multistate promoter process. Note that this stationary distribution is unique as every two states of the process $(E^\infty(t), M^\infty(t))$ are in the same communication class.

We now argue that $\mu^c$ converges to $\mu^\infty$, which is the stationary distribution of $\tilde G^\infty$.   
We have $\tilde G^c$ converges pointwise to $\tilde G$ as $c\uparrow\infty$,and that $\tilde G^c$ is banded (with respect to lexicographical ordering of states $(i,m)$). By consideration of the balance equation $\mu^c\tilde G^c = 0$, it follows that limit points $\mu_\text{lim}$ satisfy the balance equation $\mu_\text{lim}\tilde G^\infty=0$ which has unique solution $\mu^\infty$.  Hence, $\mu^c$ converges to $\mu^\infty$.

Similarly, we argue that $\pi_2^c$ converges to $\pi_2^\infty$. Specifically, let $\check G^c$ be the generator associated with the process $(E^c(t),M^c(t),P^c(t))$ for $0\le c\le\infty$. The generators $\check G^c$ are banded for an appropriate choice of ordering on states $(e,m,p)$ and converge pointwise to $\check G^\infty$. Since the sequence $\{\pi_2^c\}_{c\ge 0}$ is also tight, every limit point $\pi_{2,\text{lim}}$ of the sequence must be a distribution which satisfies $0=\pi_{2,\text{lim}}\check G^\infty$. Since $\pi_2^\infty$ is the only such distribution, $\pi_2^c$ converges to $\pi_2^\infty$.

We now consider the clumped stick-breaking construction with respected to the bounded model. 
For each value of $c<\infty$, define a Markov chain $\bf S^c$ on state space $\X\times \{0,1,\ldots, c\}$ with initial measure $\mu^c$ and non-repeating transition kernel
$$K^c_{\tilde e,\tilde f}=\frac{\tilde G_{\tilde e,\tilde f}^{c,*}}{-\tilde G_{\tilde e,\tilde e}^{c,*}}\mathbbm 1(\tilde e\neq\tilde f)+\mathbbm 1(\tilde e=\tilde f \ and\ \tilde e_2>c).$$
Note that this Markov chain only reaches states with $m\le c$.

Given $\bf S^c$, let $\bf Y^c$ be an independent sequence of $Y^c_j\sim$Beta$(1,-\tilde G^{c,*}_{S^c_j,S^c_j}/\gamma)$ variables, and ${\bf R^c} = \{Y^c_j\prod_{i=1}^{j-1}(1-Y^c_i)\}_{j\geq 1}$ be constructed from $\bf Y^c$ as a residual allocation model. Define the stick-breaking measure
$$X^c(\ \cdot \ )=\sum_{j=1}^\infty R^c_j\delta_{S^c_j}(\ \cdot\ ).$$

Now, since by Theorem \ref{cormultistatestationaryconstruction} and the clumped representation afforded by Corollary \ref{propaltMSBM}, we know that if
$$P^c|({\bf S^c,R^c})\sim\text{Poisson}\left(\frac{\alpha}{\gamma}\sum_{(e,m)}m\ X^c(e,m)\right)$$
and $S_1^c=(E^c,M^c)\sim  \mu^c$, then the stationary distribution $\pi_2^c\sim (E^c, M^c, P^c)$ for the bounded joint process can be written in terms of $\mu^c$ and the Poisson mixture $P^c|{\bf S^c, R^c}$.

We now show that one can take a limit as $c\rightarrow\infty$.
Since (a) $\mu^c$ converges pointwise to $\mu^\infty$ of $\tilde G^\infty$ and (b) the uniformly banded matrices $\tilde G^c$ converge entrywise to $\tilde G^\infty$, we have
\begin{align}
	\lim_{c\rightarrow\infty}K_{\tilde e,\tilde f}^c & =\lim_{c\rightarrow\infty}\frac{\tilde G_{\tilde e,\tilde f}^{c,*}}{-\tilde G_{\tilde e,\tilde e}^{c,*}}\mathbbm 1(\tilde e\neq\tilde f)=\lim_{c\rightarrow\infty}\frac{[D(\mu^c)^{-1}\tilde G^{c,T}D(\mu^c)]_{\tilde e,\tilde f}}{-\tilde G_{\tilde e,\tilde e}^{c}}\mathbbm 1(\tilde e\neq\tilde f)\nonumber\\
	 & =\frac{[D(\mu^\infty)^{-1}\tilde G^{\infty,T}D(\mu^\infty)]_{\tilde e,\tilde f}}{-\tilde G_{\tilde e,\tilde e}^\infty}\mathbbm 1(\tilde e\neq\tilde f)\nonumber\\
	 & =\frac{\tilde G_{\tilde e,\tilde f}^{\infty,*}}{-\tilde G_{\tilde e,\tilde e}^{\infty,*}}\mathbbm 1(\tilde e\neq\tilde f) =K_{\tilde e,\tilde f}^\infty.
	 \label{KctoK}
\end{align}

Let now $\bf S^\infty$ be a Markov chain with initial distribution $\mu^\infty$ and kernel $K^\infty$. Given $\bf S^\infty$, let $\bf Y^\infty$ be an independent sequence of $Y^\infty_j\sim$Beta$(1,-\tilde G^{\infty,*}_{S_j^\infty,S_j^\infty}/\gamma)$ variables, and $\bf R^\infty$ be constructed from $\bf Y^\infty$ as a residual allocation model.

	Let $s^n=\{s_j\}_{j=1}^n$ be a deterministic sequence of states $s_j=(e_j,m_j)$.  
	Since $\tilde G^{c,*}_{\tilde e,\tilde e}$ converges to $\tilde G^{\infty, *}_{\tilde e,\tilde e}$ as $c\uparrow\infty$, the conditional distribution of $\{R^c_j\}_{j=1}^n|\{S^c_j\}_{j=1}^n=s^n$ converges to that of $\{R^\infty_j\}_{j=1}^n|\{S^\infty_j\}_{j=1}^n=s^n$ as $c\uparrow\infty$.

	Also, as $ \mu^c$ converges to $ \mu^\infty$ and $K^c_{\tilde e,\tilde f}$ converges to $K^\infty_{\tilde e, \tilde f}$, the distribution of $\{S^c_j\}_{j=1}^n$ converges to that of
	 $\{S^\infty_j\}_{j=1}^n$ as $c\uparrow\infty$.

We conclude then, since $1=\sum_{j\geq 1}R^\infty_j = \sum_{j\geq 1} R^c_j$ for each $c$, as $c\uparrow\infty$ that 
$X^c(\cdot)$ converges weakly to 
$$X^\infty(\ \cdot\ )=\sum_{j=1}^\infty R^\infty_j\delta_{S^\infty_j}(\ \cdot\ ),$$
and $P^c|({\bf S^c, R^c})\sim {\rm Poisson}\left(\alpha\sum_{(e,m)}m\ X^c(e,m)\right)$ converges weakly to
$$P^\infty|({\bf S^\infty,R^\infty})\sim\text{Poisson}\left(\frac{\alpha}{\gamma}\sum_{(e,m)}m\ X^\infty(e,m)\right)$$
and $S_1^c=(E^c,M^c)\sim \mu^c$ converges to $S_1^\infty = (E^\infty, M^\infty)\sim \mu^\infty$.

Hence, 
$$\lim_{c\rightarrow\infty}(E^c,M^c,P^c)\ \stackrel d=\ (E^\infty,M^\infty,P^\infty)\ \sim\ \pi_2^\infty(\ \cdot\ ,\ \cdot\ ,\ \cdot\ |G,\beta,\delta,\alpha,\gamma),$$
and we conclude the proof of Theorem \ref{protein_thm}.
\qed

\section{Summary and conclusion}
\label{conclusion}

Through relations between seemingly disparate stick-breaking Markovian measures, empirical distribution limits of certain time-inhomogeneous Markov chains, and Poisson mixture representations of stationary distributions in multistate mRNA promoter models, we identify the stationary joint distribution of promoter state, and mRNA level via a constructive stick-breaking formula.  Moreover, we also consider protein interactions influenced by mRNA levels and find a stick-breaking formulation of the joint promoter, mRNA and protein levels.  Interestingly, this formula with respect to un-bounded protein levels involves a `clumped' representation of the stick-breaking measure.  These results constitute what seem to be a significant advance over previous work, which approximate stationary distributions or restrict solvable computations to specialized settings.

Importantly, the stick-breaking construction allows to sample directly from the stationary distribution, permitting inference procedures for parameters as well as model selection. Such a feature improves over sampling from the stationary distributions by running the process for a length of time. 
Our experiments show that, for various choices of the model settings, the inference procedures based on the stick-breaking construction are able to estimate model parameters accurately and select the underlying model correctly when the sample size is sufficiently large.
In addition, the form of the stationary distribution allows to compute mixed moments between mRNA and protein levels, which might bear upon correlation analysis as in \cite{1Albayrak}.

Although in principle the `stick-breaking' apparatus can be used to identify stationary distributions in linear chains of reactions, a natural problem for future work is to understand the role of `feedback' in constructing the stationary distribution in more general networks, say those where protein or mRNA levels influence promoter switching rates. 
 We have also discussed the notion of identifiability of parameters 
 and believe mRNA levels $M$ can identify the promoter switching rates $G$ and intensities $\beta$ when the $\beta=\{\beta_i\}$ components are known to be distinct.  
Our numerical results indicate that  this is the case.
Finally, of course, a next step is to understand inference of parameters and model selection from laboratory cell readings.

\medskip
{\bf Acknowledgements.}
W.L and S.S.~were partly supported by grant
ARO-W911NF-18-1-0311.

\end{document}